\documentclass[12pt]{amsart}
\usepackage[applemac]{inputenc}

\usepackage{fullpage}
\usepackage{
 ulem,
 amsfonts}
 \usepackage{amsmath}
\usepackage{amsmath,esint, amssymb, amsthm,mathbbol,upgreek,exscale,wasysym,  fontenc,siunitx}

 \usepackage{mathtools}

\usepackage[colorlinks=true, pdfstartview=FitV, linkcolor=blue, citecolor=blue,
 urlcolor=blue]{hyperref}
\usepackage{color}

\date{}

\newcommand\N{\mathbb{N}} 
\newcommand\R{\mathbb{R}} 
\newcommand\C{\mathbb{C}} 
\newcommand\T{\mathbb{T}} 
\newcommand\D{\mathbb{D}}

\newcommand\II{\num{i}}

\theoremstyle{plain}

\numberwithin{equation}{section}
\newtheorem{theorem}{Theorem}[section]
\newtheorem{proposition}[theorem]{Proposition}
\newtheorem{lemma}[theorem]{Lemma}
\newtheorem{remark}[theorem]{Remark}

\usepackage[textmath,displaymath,floats,graphics,delayed]{preview}
  \title[]{Existence of small loops in the Bifurcation diagram  near the degenerate eigenvalues}
\author[T. Hmidi]{Taoufik  Hmidi}
 \author[C. Renault]{Coralie Renault}
 
\address{IRMAR, Universit\'e de Rennes 1\\ Campus de
Beaulieu\\ 35~042 Rennes cedex\\ France}
\email{thmidi@univ-rennes1.fr}
\email{coralie.renault@univ-rennes1.fr}

\begin{document}

\begin{abstract}
In this paper we study for the incompressible Euler equations the global structure of the bifurcation diagram for the rotating doubly connected  patches near the degenerate case.  We show that the branches with the same symmetry merge forming a small loop provided that  they are close enough. This confirms the numerical observations done in the recent work \cite{3}.
 \end{abstract}

\newpage 

\maketitle{}
\tableofcontents

\section{Introduction}
During the last few decades an intensive research activity  
 has been dedicated to  the study  in fluid dynamics of  relative equilibria, sometimes called steady states or   V-states. These vortical structures have the common  feature to keep their shape without deformation during the motion and they  seem to play a central role in  the emergence of  coherent structures in  turbulent flows at large scales, see for instance  \cite{DR, Kam,Luz, Over, Saf} and the references therein. Notice that from experimental  standpoint, their existence has been revealed in different geophysical phenomena  such as the aerodynamic trailing-vortex problem, the two-dimensional shear layers, Saturn's hexagon, the K\'arm\'an vortex street, and so on.. Several numerical and analytical investigations  have been carried out in various configurations depending on the topological structure of the vortices: single simply or multiply connected vortices, dipolar or multipolar, see for instance \cite{Cor1, Cor22,03,3,H-H-H,Flierl,H-H,6,HMV2,5,H-M}.

In this paper we shall be concerned with some refined global structure of the doubly connected rotating patches for  the two-dimensional incompressible Euler equations.  These equations describe the motion of an ideal fluid  and take the  form,
\begin{equation}\label{Euler}
\left\lbrace
\begin{array}{l}
\partial_t \omega + v \cdot \nabla \omega=0, \text{ }(t,x)\in \R_+ \times \R^2, \\
v=-\nabla^\perp(-\Delta)^{-1}\omega,\\
\omega_{|t=0}=\omega_0
\end{array}
\right.
\end{equation}
 where $v=(v_1,v_2)$ refers to the velocity fields and $\omega$ being its  vorticity  which is defined by  the scalar  $\omega= \partial_1 v_2-\partial_2 v_1$.
Note that one can recover the velocity from the vorticity distribution according to the Biot-Savart law,
 \[v(x)=\frac{1}{2 \pi} \int_{\R^2} \frac{(x-y)^{\perp}}{\vert x-y \vert^2} \omega(y)dy.\]
 The global existence and uniqueness of solutions with initial vorticity lying in the space $  L^{1} \cap L^{\infty}$  is a very  classical fact  established many years ago by Yudovich \cite{Y1}. This result has the advantage to allow discontinuous vortices taking the form of vortex patches, that is $\omega_0(x)= \chi_D$ the characteristic function of a bounded   domain $D$. The time evolution of this specific structure is preserved and   the vorticity $\omega(t)$  is uniformly distributed in bounded domain $D_t$, which is nothing but the image by the flow mapping of the initial domain. The regularity of this domain is not an easy task and was solved by Chemin in \cite{2} who proved that a $C^{1+\epsilon}$-boundary keeps this regularity globally in time  without any loss. In general 
the dynamics of the boundary is hard to track and is subject to  the nonlinear  effects created by the induced velocity. Nonetheless,  some  special family of rotating patches  characterized by uniform rotation without  changing the shape are known in the literature and  a lot of implicit examples have been discovered in the last few decades. Note that in this setting   we have explicitly  $D_t=R_{0,\Omega t }D$ where $R_{0,\Omega t }$ is a planar rotation centered at the origin and  with angle $\Omega t$; for the sake of simplicity we  have assumed that the center of rotation is the origin of the frame and the parameter $\Omega$ denotes  the angular velocity of the rotating domains. The first example was discovered very earlier by  Kirchhoff in \cite{Kirc} who showed that  an ellipse of semi-axes $a$ and $b$ rotates about its center uniformly with the angular velocity $\Omega= \frac{ab}{(a^2+b^2)}$. Later, Deem and Zabusky gave in \cite{03} numerical evidence of the existence of the V-states with $m-$fold symmetry for the integers $m\in \{3,4,5\}$.  Few years after, Burbea gave in \cite{1} an analytical proof of the existence  using complex analysis formulation and bifurcation theory. The regularity of the V-states close to Rankine vortices was discussed quite recently  in \cite{Cor1,6}.   We point out that the bifurcation from the ellipses was studied numerically and analytically in \cite{Cor22,H-M,Kam}. All these results are restricted to simply connected domains and the analytical investigation of doubly connected V-states has been initiated with the works \cite{3,HMV2}. To fix the terminology,  a domain $D$ is said   doubly connected if it takes the form  $D=D_1 \setminus D_2$ with $D_1$ and $D_2$  being two simply connected bounded domains satisfying $\overline{D_2}\subset D_1$. The main result of 
\cite{3} which is  deeply connected to the aim of this paper deals with the bifurcation from the annular patches  where  $D=\mathbb{A}_b\equiv \lbrace z; b< \vert z \vert <1 \rbrace$. For the clarity of the discussion we shall recall the main result of \cite{3}.
\begin{theorem}\label{eulerdoubly}
Given $b \in (0,1)$ and  let $m\geq3$ be a positive integer such that,
\begin{equation}\label{deltaX}1+b^m-\frac{(1-b^2)}{2}m<0.
 \end{equation}
Then there exist two curves of doubly connected rotating patches with m-fold symmetry bifurcating from the annulus $\mathbb{A}_b$ at the angular velocities,
 \begin{equation*}
 \Omega_m^{\pm}=\frac{1-b^2}{4} \pm \frac{1}{2m}\sqrt{\Delta_m}
 \end{equation*}
 with
\[
 \Delta_m= \left(\frac{1-b^2}{2}m-1\right)^2 -b^{2m} .
\]

\end{theorem}
We emphasize that the condition \eqref{deltaX}  is required by the transversality assumption, otherwise the eigenvalues $\Omega_m^\pm$  are double and thus the classical theorems in the bifurcation theory such as  Crandall-Rabinowitz theorem \cite{CR} are out of use. The analysis of the  degenerate case corresponding to vanishing discriminant ( in  which case $\Omega_m^+=\Omega_m^-)$  has been explored very recently in \cite{5}.  They proved in particular  that for $m \geq 3$ and $b\in (0,1)$ such that $\Delta_m=0$ there is no bifurcation to $m$-fold V-states. However for $b\in (0,1)\backslash \mathcal{S}$ two-fold V-states still bifurcate from the annulus $\mathbb{A}_b$ where $\mathcal{S}=\{ b_m^*, m\geq3\}$ and $b_m^*$ being the unique solution in the interval $(0,1) $ of the equation
\begin{equation}\label{master1} 1+b^m-\frac{(1-b^2)}{2}m=0. 
\end{equation}
The proof of this result is by no means  non trivial and based on the local structure of the reduced bifurcation equation obtained through the use of Lyapunov-Schmidt reduction. 
Note that according to the numerical experiments done in \cite{3}  two different scenarios for global bifurcation are conjectured. The first one when the eigenvalues $\Omega_m^-$ and $\Omega_m^+$ are far enough  in which case each  branch ends with a singular V-state and the singularity is a corner of \mbox{angle $\frac\pi2$.} For more details about  the  structure of the limiting V-states we refer the reader to  the  Section $9.3$ in \cite{3}.  Nevertheless, in the second scenario where the eigenvalues are close enough there is no singularity formation on the boundary and it seems quite  evident  that there is no spectral gap  and the V-states can be constructed for any $\Omega\in [\Omega_m^-, \Omega_m^+]$, see Fig. \ref{Fig1} taken from \cite{3}. Moreover,  drawing the second Fourier coefficient of each  conformal mapping that parametrize each boundary as done  in \mbox{Fig. \ref{Fig2}} we get a small loop passing through the trivial solution at $\Omega_m^-$ and $\Omega_m^+.$  This suggests  that the  bifurcation curve starting from $\Omega_m^+$  will  return  back   to the trivial solution (annulus) at $\Omega_m^-$.  

 Our main purpose  in this paper is to go further in this study by  checking analytically the second scenario and provide for $m\geq 3$ the global structure of  the bifurcation curves  near the degenerate case. Our  result reads as follows.
\begin{figure}[t!]
\center
\includegraphics[width=0.4\textwidth, clip=true]{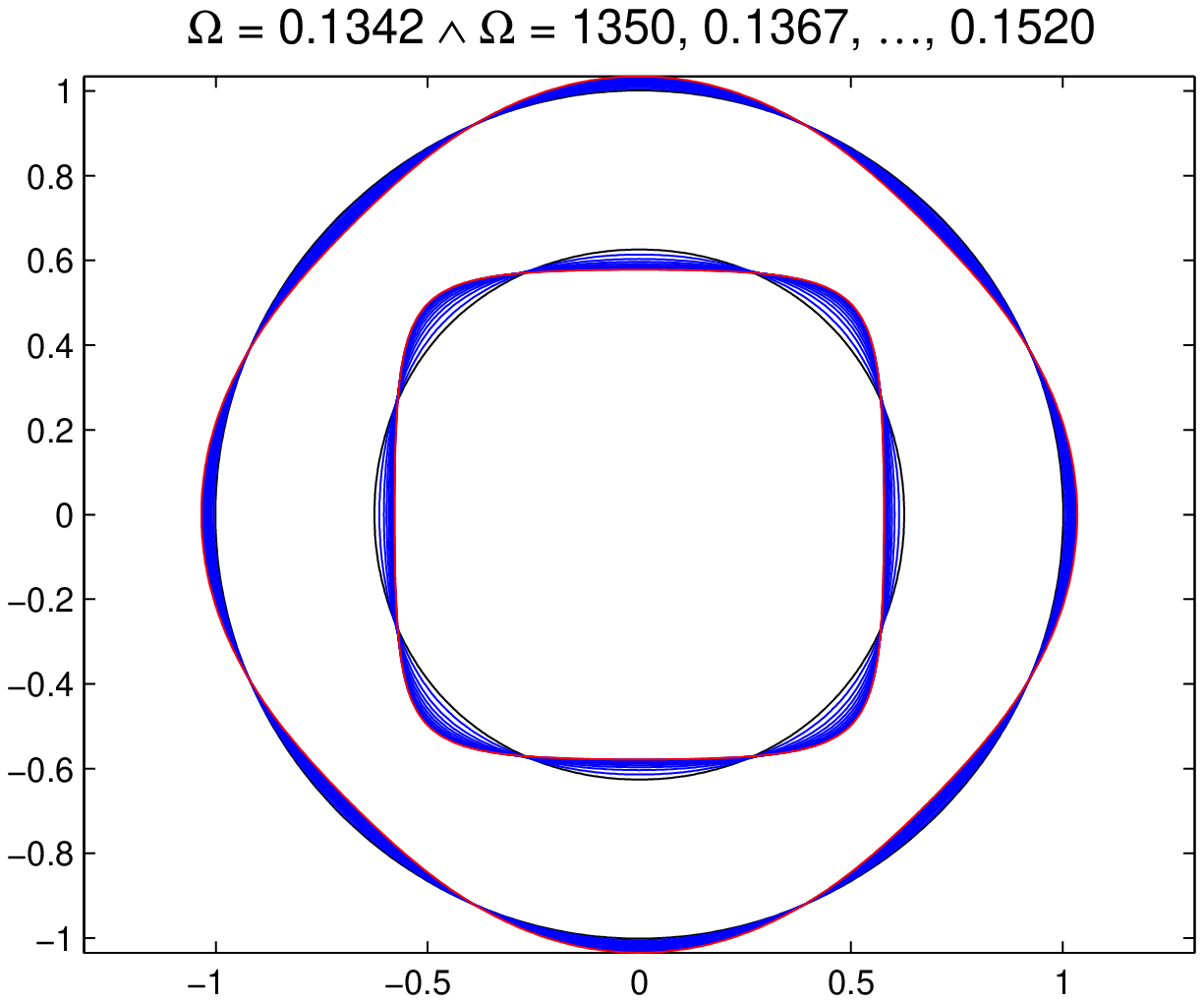}~\includegraphics[width=0.4\textwidth, clip=true]{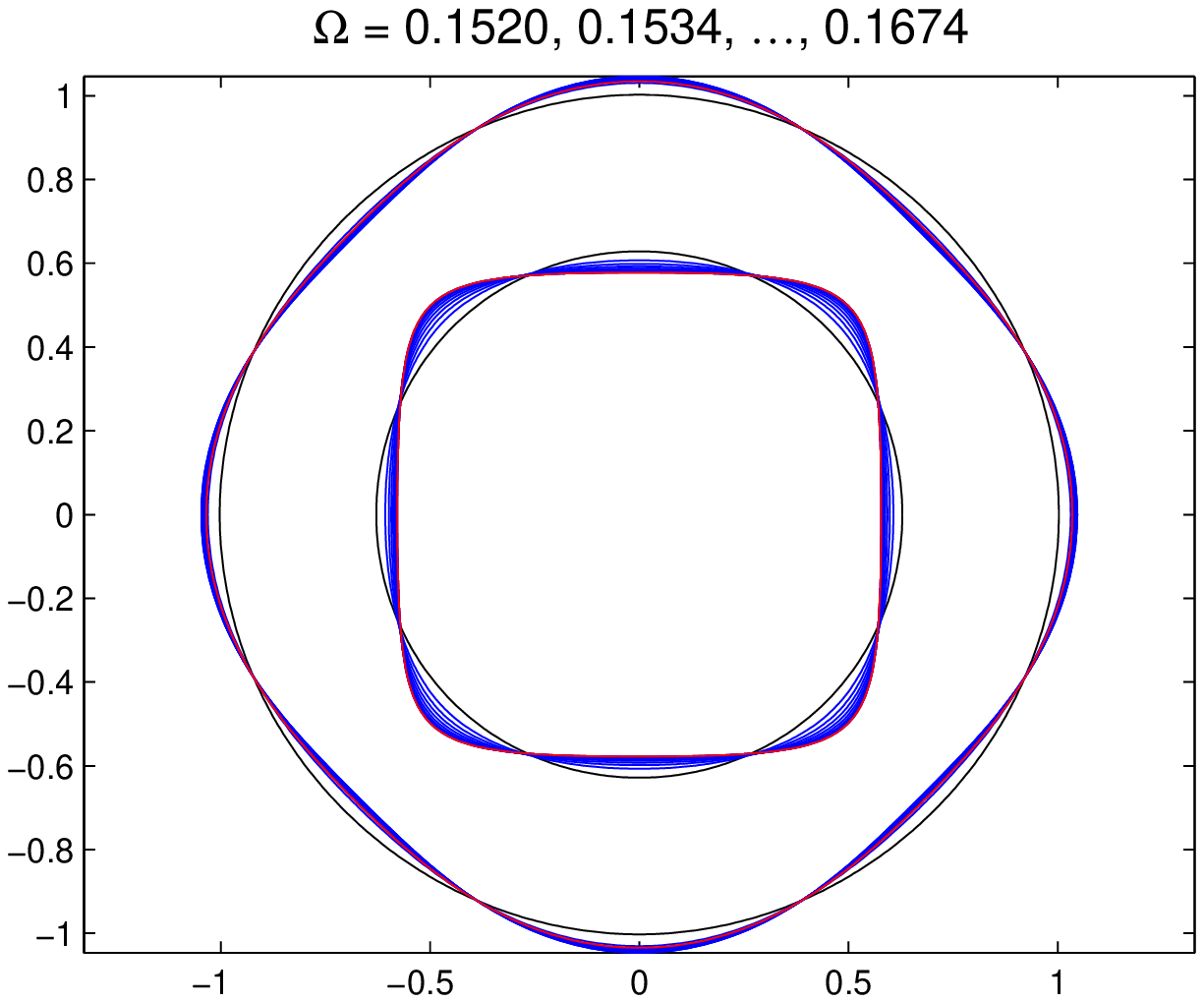}
\caption{Family of 4-fold $V$-states, for $b = 0.63$ and different
$\Omega$. We observe that there is no singularity formation in the boundary.} \label{Fig1}
\end{figure}

\begin{figure}[!htb]
\center
\includegraphics[width=0.4\textwidth, clip=true]{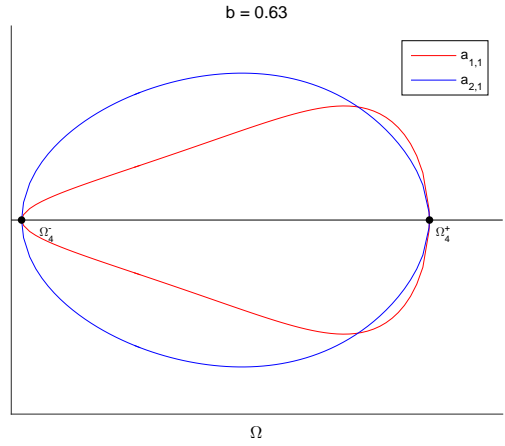}
\caption{\small{Bifurcation curves of the Fourier coefficients $a_{1,1}$ and $a_{2,1}$ in \eqref{conformal1} with respect to $\Omega$.}}
\label{Fig2}
\end{figure}
 
 \begin{theorem}\label{main}
Let  $m\geq 3$ and  $b_m^*$ be the unique solution in $(0,1) $ of the equation \eqref{master1}.
Then there exists $ b_m\in(0,b_m^*)$ such that for any $b$ in $ (b_m,b^*_m)$ the two  curves of $m$-fold V-states given by  Theorem $\ref{eulerdoubly}$ merge and form a  loop.

\end{theorem}
Now, we are going to outline the main steps of the  proof. Roughly speaking we start with writing down the equations governing the boundary  $\partial D_1\cup \partial D_2$ of the rotating patches and attempt to follow the approach developed in \cite{3}.  For $j\in\{1,2\}$, let $\Phi_j : \D^c  \rightarrow D_j^c$ be the conformal mapping which enjoys  the following structure,
\begin{equation}\label{conformal1}
\forall\, \vert z \vert \geq 1 \text{, }\Phi_j(z)=b_jz+\sum_{n\in \N} \frac{a_{j,n}}{z^{n}} \text{, } a_{j,n} \in \R \text{, } b_1=1\quad \text{ and } \quad b_2=b.
\end{equation}
 We have denoted by $\D^c$ the complement of the open unit disc $\D$ and we have also assumed  that the Fourier coefficients of the conformal mappings are real which means that we look for  V-states which  have at least one axis of symmetry that can be chosen to be the real axis. It is also important to mention that the domain $D$ is implicitly assumed to be smooth enough, more than $C^1$ as we shall see in the proof, and therefore   each conformal mapping can be extended  up to the boundary. According to the  subsection $\ref{boundary equation}$ the conformal mappings satisfy the coupled equations: for $j \in \lbrace 1,2 \rbrace$
\[G_j(\lambda,f_1,f_2)(w)  \triangleq \textnormal{Im} \left\lbrace \left( (1-\lambda)\overline{\Phi_j(w)}+I(\Phi_j(w))\right) w \Phi_j'(w) \right\rbrace=0  \, ,\forall w \in \T \]
where
\[  \lambda= 1-2 \Omega,\quad \Phi_j(w)=b_j w+f_j({{w}})\]
and
\[I(z) =\frac{1}{2 \num{i}\pi}\int_{\T} \frac{\overline{z}-\overline{\Phi_1(\xi)}}{z-\Phi_1(\xi)} \Phi_1'(\xi) d\xi -\frac{1}{2 \II\pi} \int_{\T} \frac{\overline{z}-\overline{\Phi_2(\xi)}}{z-\Phi_2(\xi)} \Phi_2'(\xi) d\xi.\]
 Here $d\xi$ denotes the complex integration over the unit circle $\T$.
The linearized operator around  the annulus defined through 
\[  \mathcal{L}_{\lambda,b}(h)\triangleq \partial_fG(\lambda,0)h=\frac{d}{dt} [G(\lambda,th)]_{|t=0} \]
 plays a significant role in the proof and according to   \cite{5} it acts as a matrix Fourier multiplier. Actually, 
for  $h=(h_1,h_2)$ chosen  in suitable Banach space with
\[h_j(w)=\sum_{n\geq 1} \frac{a_{j,n}}{w^{nm-1}},\, a_{j,n}\in \R\]
we have the expression
\[ \mathcal{L}_{\lambda,b}(h) =\sum_{n\geq 1}M_{nm}(\lambda) \left(\begin{array}{c}
a_{1,n} \\ 
a_{2,n}
\end{array} \right) e_{nm} \quad \text{ with }\quad e_n(w)= \textnormal{Im}(\overline{w}^n)\]
where for $ n\geq 1$ the matrix $M_n$ is given by
\[ M_n(\lambda)= \left( \begin{array}{cc}
n\lambda-1-nb^2 & b^{n+1} \\ 
-b^n & b(n\lambda-n+1)
\end{array} \right).\]
It is known from \cite{3} that for given $m\geq3$ the values of $b$ such that $M_m(\lambda)$ is singular, for suitable values of $\lambda=\lambda_m^\pm$,  belong to the interval  $(0, b_m^*)$ where $b_m^*$ has been introduced in \eqref{master1}. It is also  shown in that paper that the assumptions of Crandall-Rabinowitz theorem are satisfied, especially the transversality assumption which reduces the bifurcation study to some properties of the linearized operator.  This latter property is no longer true for $b=b_m^*$ and we have double eigenvalues $\lambda_m^\pm=\frac{1+b_m^*}{2}$. This is a degenerate case and we know from  \cite{5} that   there is no bifurcation. It seems that the approach implemented  in this situation can be carried out for  $b\in (0,b_m^*)$ but  close enough to $b_m^*$. 
In fact,  using   Lyapunov-Schmidt reduction (through appropriate projections) we  transform the infinite-dimensional problem into a two dimensional one. Therefore  the V-states equation reduces to the resolution of an  equation   of the type
\[ F_2(\lambda ,t )=0\quad  \text{ with } \quad F_2 : \R^2 \rightarrow \R , \]
with $F_2$ being a smooth function and note that when  $b=b_m^*$ the point $(\lambda_m^\pm,0)$ is a critical point for $F_2$ and for that reason one should expand $F_2$ to the second order around this point  in order to understand the resolvability of the reduced equation. At the order two $F_2$ is strictly convex and therefore locally the critical point is the only solution for $F_2.$ Reproducing this approach  in the current setting and  after  long and involved  computations we find that for $(\lambda,t)$ close enough, for example,  to the solution $(\lambda_m^+,0)$ 

\[ F_2(\lambda,t)=a_m(b) (\lambda-\lambda_m^+)+c_m(b) (\lambda-\lambda_m^+)^2+d_m(b) t^2+\big((\lambda-\lambda_m^+)^2+t^2\big)\varepsilon(\lambda,t)\]
with
\begin{equation*}
\underset{(\lambda,t) \rightarrow (\lambda_m^+,0)}{\lim} \varepsilon(\lambda,t)=0.
\end{equation*}
Notice that
$$
a_m(b_m^*)=0, \quad  c_m(b_m^*)>0,\quad  d_m(b_m^*)>0
$$
and moreover   for $b$ belonging to a small interval  $(b_m, b_m^*)$ we get $a_m(b)>0.$ 
As we can easily check from the preceding facts, the zeros of the associated quadratic form of $F_2$ is  a small ellipse. Therefore by perturbation arguments, we may show that the solutions  of $F_2$ are actually a perturbation in a strong topology of the ellipse. Consequently the solutions of $F_2$ around the point $(\lambda_m^+,0)$ can be parametrized  by a  smooth  Jordan curve.  We remark that in  addition to the known  trivial solution $(\lambda_m^+,0)$ one can find  a second one, of course different from the preceding one,  in the form $(\lambda_m,0)$, corresponding geometrically to the same annulus $\mathbb{A}_b$. Around this point $(\lambda_m,0)$ the obtained curve describes a bifurcating curve of V-states with exactly  $m-$fold symmetry and therefore  from the local description of the bifurcation diagram stated  in Theorem \ref{eulerdoubly}  we deduce that $\lambda_m=\lambda_m^-$. This means the formation of loops as stated in our  main theorem.

The paper is organized as follows. In Section \ref{Remind} we shall introduce some tools, formulate the V-states equations and write down the reduced bifurcation equation using Lyapunov-Schmidt reduction. Section \ref{bifurcationnearthedegeneratecase} is devoted to Taylor expansion at order two of the reduced  bifurcation equation and the complete proof of Theorem \ref{main} will be given  in the last section.

\section{Reminder and preliminaries}\label{Remind}
We shall recall in this section some tools that we shall frequently use throughout the paper and write down the reduced bifurcation equation which is the first step towards the proof of \mbox{Theorem \ref{main}.} But before that we  will  fix  some notations. The unit disc and its boundary will be denoted respectively by $\D$ and $\T$.
and $D_r$  is the planar disc of radius $r$ and centered at the origin. Given 
a continuous function  $f: \T \rightarrow \C$ we define its mean value by,
\[ \fint_{\T} f(\tau) d\tau \triangleq \frac{1}{2\II \pi} \int_{\T} f(\tau) d\tau, \]
where $d\tau$ stands for the complex integration.
\subsection{H\"{o}lder  spaces}
Now we shall introduce  H\"older spaces  on the unit circle $\mathbb{T}$. Let $0<\gamma<1$ we denote by $C^\gamma(\mathbb{T}) $  the space of continuous functions $f$ such that
$$
\Vert f\Vert_{C^\gamma(\mathbb{T})}\triangleq \Vert f\Vert_{L^\infty(\mathbb{T})}+\sup_{\tau\neq w\in \mathbb{T}}\frac{\vert f(\tau)-f(w)\vert}{\vert \tau-w\vert^\alpha}<\infty.
$$
For any integer $n$, the space $C^{n+\gamma}(\mathbb{T})$ stands for the set of functions $f$ of class $C^n$ whose $n-$th order derivatives are H\"older continuous  with exponent $\gamma$. It is equipped with the usual  norm,
$$
\Vert f\Vert_{C^{n+\gamma}(\mathbb{T})}\triangleq \Vert f\Vert_{L^\infty(\mathbb{T})}+\Big\Vert \frac{d^n f}{dw^n}\Big\Vert_{C^\gamma(\mathbb{T})}.
$$ 

Recall that for $n\in\N$, the space $C^n(\T)$ is the set of functions $f$ of class $C^n$ such that,
$$
\|f\|_{C^n(\mathbb{T})}\triangleq \sum_{k=0}^n\|f^{(k)}\|_{L^\infty(\T)}<\infty.
$$

\subsection{Boundary equations}\label{boundary equation}
Let $D_2 \Subset D_1$ be two simply connected domains and $D=D_1\setminus D_2$ be a doubly connected domains. The boundary of $D_j$ will be denoted by $\Gamma_j$. Then according to  \cite{3, 5} we find that that the exterior  conformal mappings $\Phi_1$ and $\Phi_2$ associated to $D_1$ and $D_2$ satisfy the coupled nonlinear equations. For $j\in\{1,2\},$ 
\[\widehat{G}_j(\lambda,\Phi_1,\Phi_2)(w)=0,\quad  \forall \, w \in \T, \]
with
\begin{equation}\label{boundary}
\widehat{G}_j(\lambda,\Phi_1,\Phi_2)(w)  \triangleq \textnormal{Im} \left\lbrace \left( (1-\lambda)\overline{\Phi_j(w)}+I(\Phi_j(w))\right) w \Phi_j'(w) \right\rbrace. 
\end{equation}
 Note that we have introduced $\lambda\triangleq 1-2 \Omega$ because it  is more convenient for the computations and
\[I(z) =\fint_{\T} \frac{\overline{z}-\overline{\Phi_1(\xi)}}{z-\Phi_1(\xi)} \Phi_1'(\xi) d\xi - \fint_{\T} \frac{\overline{z}-\overline{\Phi_2(\xi)}}{z-\Phi_2(\xi)} \Phi_2'(\xi) d\xi.\]
The integrals are defined in the complex sense and we shall  focus on V-states which are small perturbation of the annulus $\mathbb{A}_b=\big\{z, b\le|z|\le 1\big\}$ with $b\in (0,1)$. The conformal mappings  $\Phi_j$ with $j\in\{1,2\}$ admit the expansions,
\[\forall |z|\geq 1,\quad \Phi_j(z)=b_jz+f_j(z)=z + \sum_{n=1}^{+ \infty} \frac{a_{j,n}}{z^n}  \]
with 
\[b_1=1, b_2=b. \]
Define 
\begin{equation}\label{V-state}
G_j(\lambda,f_1,f_2)\triangleq \widehat{G}_j(\lambda,\Phi_1,\Phi_2)
\end{equation}
then the equations of the V-states become,
\[\forall w \in \T \text{, } G(\lambda,f_1,f_2)(w)=0 \]
with
\[G=(G_1,G_2).\]
Note that the annulus is a solution for any angular velocity, that is, 
\[ G(\lambda,0,0)=0 \]
and the set $\lbrace (\lambda,0,0) | \lambda \in \R \rbrace $ will be called  the set of trivial solutions.
\subsection{Reduced bifurcation equation.}\label{Linearized operator and Lyapunov-Schmidt reduction}
For any integer $m \geq 3$,  the existence of V-states was proved  in  \cite{3} provided $b\in(0,b_m^*)$ that guarantee the transversality assumption. The idea is to check  that the functional $G$ has non trivial zeros   using bifurcation arguments. However, and as we have mentioned before in the Introduction, the knowledge  of the  linearized operator around the trivial solution is not enough to understand the structure of the bifurcating curves near the degenerate case corresponding to double eigenvalues. To circumvent this difficulty we make an expansion at order two of the reduced bifurcation equation in the spirit of \cite{5}, and this will be  the subject of the current task. Let us first introduce Banach spaces that we shall use and recall the algebraic structure of the linearize operator.
%
For $\alpha\in (0,1)$, we set
 \[X_m= \left\lbrace f=(f_1,f_2) \in (C^{1+\alpha}(\T))^2, f(w)=\sum_{n=1}^{+ \infty} A_n \overline{w}^{nm-1}, A_n \in \R^2 \right\rbrace. \]
 and
 \[Y_m= \left\lbrace G=(G_1,G_2) \in (C^{\alpha}(\T))^2, G=\sum_{n=1}^{+ \infty} B_ne_{nm}, B_n \in \R^2 \right\rbrace \text{, } e_n(w)= \textnormal{Im}(\overline{w}^n). \]
Note that the domains $D_j$ whose  conformal mappings $\Phi_j$ associated to the perturbations $f_j$ lying in $X_m$ are actually  $m-$fold symmetric.
Recall from the subsection \ref{boundary equation} that  the equation of $m-$fold symmetric V-states is given by
\begin{equation}\label{V-state}
G(\lambda,f)=0 \text{, }\quad f=(f_1,f_2)\in B_r^m \times B_r^m \subset X_m
\end{equation}
where $B_r^m$ is the ball given by
$$
B_r^m=\Big\{ f\in C^{1+\alpha}(\mathbb{T}),\, f(w)=\sum_{n=1}^{\infty}a_{n}\overline{w}^{nm-1},\, a_n\in \mathbb{R}, \|f\|_{C^{1+\alpha}}\leq r\Big\}.
$$ 
We mention that we are looking for solutions close to the trivial solutions and therefore the \mbox{radius $r$} will be taken small enough.
The linearized operator around zero is defined by
\[   \partial_fG(\lambda,0)h=\frac{d}{dt} [G(\lambda,th)]_{|t=0}. \]
As it is proved in \cite{5}, for $h=(h_1,h_2)\in X_m$ taking the expansions
\[h_j(w)=\sum_{n\geq 1} \frac{a_{j,n}}{w^{nm-1}},\]
we get the expression
\begin{equation}\label{lin11}
\partial_fG(\lambda,0)h =\sum_{n\geq 1}M_{nm}(\lambda) \left(\begin{array}{c}
a_{1,n} \\ 
a_{2,n}
\end{array} \right) e_{nm},
\end{equation}
where for $ n\geq 1$ the matrix  $M_n$ is given by
\[ M_n(\lambda)= \left( \begin{array}{cc}
n\lambda-1-nb^2 & b^{n+1} \\ 
-b^n & b(n\lambda-n+1)
\end{array} \right).\]
We say throughout this paper  that $\lambda$ is an eigenvalue  if for some $n$ the  matrix $M_n(\lambda)$ is not invertible. Since
\[ \textnormal{det}(M_n(\lambda))=(n\lambda-1-nb^2)b(n\lambda-n+1)+ b^{2n+1}\]
 is a polynomial of second order on the variable $\lambda$, the roots are real if and only if its discriminant is positive. From  \cite{5}  we remind that the roots take the form,
 \[\lambda_n^{\pm}=\frac{1+b^2}{2} \pm \frac{1}{n}\sqrt{\Delta_n(b)}\]
 with the constraint 
 \[\Delta_n(b)= \Big(\frac{1-b^2}{2}n-1\Big)^2 -b^{2n}\geq 0.\] 
 According to \cite{3} this condition is equivalent for $n\geq3$ to
 \begin{equation}\label{equiv11}
n\frac{1-b^2}{2} -1\geq b^n.
\end{equation}
In addition, it is also proved that for any integer $ m \geq 3$ there exists a unique  $ b_m^* \in (0,1)$ such that  $\Delta_m(b_m^*)=0$ and $\Delta_m(b)>0$ for all $b\in [0,b_m^*)$. Moreover,
 \[ \textnormal{Ker}( \partial_fG(\lambda_m^{\pm},0))= \langle v_m \rangle\]
 with 
 \[v_m(w)= \left(\begin{array}{c}
 \frac{m\lambda_m^{\pm}-m+1}{ b^{m-1}} \\ 
1
 \end{array} \right)\overline{w}^{m-1} \triangleq \left( \begin{array}{c}
 v_{1,m} \\ 
 v_{2,m}
 \end{array} \right)\overline{w}^{m-1}. \]
 In order to be  rigorous we could   write $v_m^{\pm}$ but for the sake of simple notations we note simply $v_m$. 
Now we shall   introduce  a complement $\mathcal{X}_m$ of  the subspace $ \langle v_m \rangle$ in the space $X_m$,
\[ \mathcal{X}_m= \left\lbrace h\in (C^{1+\alpha}(\T))^2, h(w)=\sum_{n=2}^{+ \infty} A_n \overline{w}^{nm-1}+ \alpha \left( \begin{array}{c}
 1 \\ 
 0
 \end{array} \right) \overline{w}^{m-1}, A_n \in \R^2 , \alpha \in \R \right\rbrace .\]
It is easy to prove that the subspace is closed  and 
  \[ X_m = \langle v_m \rangle \oplus \mathcal{X}_m.\]
  In addition  the range $ \mathcal{Y}_m$ of $\partial_fG(\lambda_m^{\pm},0)$ in $Y_m$ is given by
 
 \[\mathcal{Y}_m=\left\lbrace K \in (C^{\alpha}(\T))^2, K=\sum_{n=2}^{+ \infty} B_ne_{nm}+\beta \left(\begin{array}{c}
 b^m \\ 
 m\lambda_m^{\pm}-m+1
 \end{array} \right)e_m, B_n \in \R^2 ,\beta \in  \R \right\rbrace . \]
 The subspace $\mathcal{Y}_m$ is of co-dimension one and its complement is a line generated by
 \begin{align*}
 \mathbb{W}_m&= \frac{1}{\sqrt{( m\lambda_m^{\pm}-m+1)^2+b^{2m}}} \left( \begin{array}{c}
 m\lambda_m^{\pm}-m+1 \\ 
 -b^m
 \end{array} \right) e_m\\
 &\triangleq \widehat{\mathbb{W}}_me_m.
 \end{align*}
 Thus we have 
 \[ Y_m=\langle \mathbb{W}_m \rangle \oplus \mathcal{Y}_m .\]
 Lyapunov-Schmidt reduction relies on  two projections
 \[ P : X_m \rightarrow \langle v_m \rangle,\quad 
 Q: Y_m \rightarrow \langle \mathbb{W}_m \rangle. \]
For a future use we need the explicit expression of the projection $Q.$ 
The Euclidian scalar product of $\R^2$ is denoted by $\langle \, , \, \rangle$ and  for $h \in Y_m$ we have
\[ h =\sum_{n=1}^{+ \infty} B_n e_{nm} \text{, } \quad Qh(w)= \langle B_1, \widehat{\mathbb{W}}_m\rangle \mathbb{W}_m.\]
Moreover, by the definition of $Q$ one has,
\begin{equation}\label{Qnul}
Q \partial_f G(\lambda_m^{\pm},0)=0.
\end{equation}
Unlike the degenerate case, the transversality assumption holds true
\[\partial_{\lambda} \partial_f G(\lambda_m^{\pm},0)v_m \notin \textnormal{Im}(\partial_f G(\lambda_m^{\pm},0))\]
and therefore
\begin{equation}\label{id-q}
Q\partial_{\lambda} \partial_f G(\lambda_m^{\pm},0)v_m\neq 0.
\end{equation}
For $f \in X_m$ we use the decomposition 
\[f=g+k\quad \text{with } \quad g=Pf \quad\text{ and } \quad k=(\text{Id}-P)f.\]
Then the V-state equation is equivalent to the system
\[ F_1(\lambda,g,k) \triangleq (\text{Id}-Q)G(\lambda,g+k)=0 \quad \text{ and } \quad QG(\lambda, g+k)=0. \]
Note that  $F_1: \R  \times \langle v_m \rangle \times \mathcal{X}_m \rightarrow \mathcal{Y}_m$  is well-defined and smooth. Thus  using $(\ref{Qnul})$ we can check the identity,
\[ D_kF_1(\lambda_m^{\pm},0,0)=(\text{Id}-Q)\partial_fG(\lambda_m^{\pm},0)=\partial_fG(\lambda_m^{\pm},0).\]
Consequently
\[  D_kF_1(\lambda_m^{\pm},0,0): \mathcal{X}_m \rightarrow \mathcal{Y}_m\]
 is invertible. The inverse is explicit and is given by the formula
 \begin{equation}\label{Invert11}
\partial_fG(\lambda_m^\pm,0)h=K\Longleftrightarrow \forall n\geq2, A_n=M_{nm}^{-1} B_n\quad  \hbox{and}\quad \alpha=-\frac{\beta}{b^m}(m\lambda_m^\pm-m+1)\cdot
\end{equation}
\label{voisinage} Thus using the implicit function theorem,  the solutions of the equation $F_1(\lambda,g,k)=0$ are locally described around the point  $(\lambda_m^{\pm},0) $ by the parametrization $k=\varphi(\lambda,g)$ with
\[ \varphi : \R \times \langle v_m \rangle \rightarrow \mathcal{X}_m. \]
being a smooth function. Remark that in principle $\varphi$ is locally defined but it can be extended globally to a smooth function still denoted by $\varphi$. Moreover, the resolution of the V-state equation near to $(\lambda_m^{\pm},0) $ is equivalent to 
\begin{equation}\label{F1voisinage}
 QG\big( \lambda , tv_m+ \varphi(\lambda, tv_m)\big)=0.
 \end{equation}
As $G(\lambda,0)=0$, $\forall \lambda$  it follows
\begin{equation}\label{phitriv}
 \varphi (\lambda,0) =0 \text{, } \forall \lambda \in \mathcal{V}(\lambda_m^{\pm}),
\end{equation}
where  $\mathcal{V}(\lambda_m^{\pm})$ is a small  neighborhood of $\lambda_m^{\pm}$.
Using Taylor expansion at order $1$ on the variable $t$ the V-states equation \eqref{F1voisinage} is equivalent to the {\it reduced bifurcation equation}, 
\begin{equation}\label{be}
 F_2(\lambda,t) \triangleq \int_0^1 Q \partial_fG(\lambda,stv_m+ \varphi(\lambda, stv_m))(v_m+ \partial_g \varphi(\lambda, stv_m)v_m) ds=0.
\end{equation}
In addition, using $(\ref{Qnul})$ we remark that
\[ F_2(\lambda_m^{\pm},0)=0.\]
\section{Taylor expansion}\label{bifurcationnearthedegeneratecase}
The goal of this section is to compute Taylor expansion of   $F_2$ at the second order. This quadratic form will answer about the local structure of the solutions of the equation \eqref{be}.
\subsection{General formulae}
The aim of this paragraph is to provide some general results concerning the first and second  derivatives of $\varphi$ and  $F_2.$ First
notice that the transversality assumption required for Crandall-Rabinowitz theorem  is given by 
\[ \partial_\lambda F_2(\lambda_m^{\pm},0)=Q\partial_{\lambda} \partial_f G(\lambda_m^{\pm},0)v_m\neq 0.\]
Thus applying the implicit function theorem to $F_2$, we get in a small neighborhood of $(\lambda_m^{\pm},0)$ a unique curve of solutions $t\in [-\varepsilon_0,\varepsilon_0]\mapsto (\lambda(t),t)$ . We shall prove that for $b$ close enough to $b_m^*$ this curve a smooth Jordan curve and for this aim we need to know the full structure of the quadratic form associated to $F_2.$  The following identities  were proved in  \cite[p 13-14]{5}.
\begin{equation}\label{partialphi}
 \partial_{\lambda} \varphi(\lambda_m^{\pm},0)= \partial_g \varphi(\lambda_m^{\pm},0)v_m=0 .
\end{equation}
and 
\begin{equation}\label{doublephi}
\partial_{\lambda \lambda} \varphi( \lambda_m^{\pm},0)=0 .
\end{equation}
Now, we give the expressions of the  coefficients of the quadratic form associated to $F_2$ around the point $(\lambda_m^\pm,0).$ For the proof see \cite[Proposition 2]{5}. 
\begin{proposition}\label{ImplicitSS}
 The following assertions hold true. 
\begin{enumerate}
\item First derivatives:
\[\partial_\lambda F_2(\lambda_m^{\pm},0)= Q \partial_\lambda \partial_f G(\lambda_m^{\pm},0)(v_m)\]
and
\begin{align*}
\partial_t F_2( \lambda_m^{\pm},0) &= \frac{1}{2} Q \partial_{ff} G (\lambda_m^{\pm},0)[v_m,v_m] \\
&=\frac{1}{2}\frac{d^2}{dt^2}[QG(\lambda_m^{\pm},tv_m)]_{|t=0} .
\end{align*}
\item Expression of $\partial_{\lambda \lambda} F_2(\lambda_m^{\pm},0)$:
$$
\partial_{\lambda \lambda} F_2(,0) =-2Q\partial_{\lambda}\partial_fG(\lambda_m^{\pm},0)[\partial_f G(\lambda_m^{\pm},0)]^{-1} (\textnormal{Id}-Q)\partial_{\lambda} \partial_f G(\lambda_m^{\pm},0)v_m
$$

\item Expression of  $\partial_{tt} F_2(\lambda_m^{\pm},0)$:
\[\partial_{tt} F_2(\lambda_m^{\pm},0)= \frac{1}{3} \frac{d^3}{dt^3}[QG(\lambda_m^{\pm},tv_m)]_{| t=0}+ Q\partial_{ff}G(\lambda_m^{\pm},0)[v_m,\widehat{v}_m] \]
with
\begin{align*}
\widehat{v}_m\triangleq 
& \frac{d^2}{dt^2} \left. \varphi(\lambda_m^{\pm},tv_m)\right|_{t=0}\\
=&-[\partial_f G(\lambda_m^{\pm},0)]^{-1} \frac{d^2}{dt^2}  [(\textnormal{Id}-Q)G(\lambda_m^{\pm},tv_m)]_{|t=0}
\end{align*}
and
\[Q\partial_{ff}G(\lambda_m^{\pm},0)[v_m,\widehat{v}_m] =\partial_t\partial_s[ QG(\lambda_m^{\pm}, tv_m+ s\widehat{v}_m)]_{|t=0,s=0}.\]
\item  Expression of  $\partial_{\lambda} \partial_t F_2(\lambda_m^{\pm},0)$:
\begin{align*}
\partial_{\lambda} \partial_t F_2(\lambda_m^{\pm},0)&=\frac{1}{2}Q\partial_{\lambda} \partial_{ff} G(\lambda_m^{\pm},0)[v_m,v_m]+ \frac{1}{2}Q\partial_{\lambda}\partial_f G(\lambda_m^{\pm},0)(\widehat{v}_m)\\
&+ Q\partial_{ff}G(\lambda_m^{\pm},0)[v_m,\partial_{\lambda}\partial_g \varphi( \lambda_m^{\pm},0)v_m]
\end{align*}
with
$$
\partial_{\lambda}\partial_g \varphi( \lambda_m^{\pm},0)v_m=-[\partial_f G(\lambda_m^{\pm},0)]^{-1} (\textnormal{Id}-Q)\partial_{\lambda} \partial_f G(\lambda_m^{\pm},0)v_m
$$
\end{enumerate}
\end{proposition}
\subsection{Explicit formula for the quadratic form}\label{jacobianandhessioncomputation}
In this section we want to explicit the terms in the Taylor expansion of $F_2$ at the second order. The main result reads as follows.

\begin{proposition}\label{explicit}
Let $m\geq 3$ and  $b\in (0,b^*_m)$. Then the following assertions hold true.
\begin{enumerate}
\item Expression of $\partial_t F_2( \lambda_m^{\pm},0)$.
\[\partial_t F_2( \lambda_m^{\pm},0)=0.\]
\item Expression of $\partial_\lambda F_2(\lambda_m^{\pm},0)$.
\[\partial_\lambda F_2(\lambda_m^{\pm},0) = \frac{m[(m\lambda_m^{\pm}-m+1)^2-b^{2m}]}{b^{m-1}[(m\lambda_m^{\pm}-m+1)^2+b^{2m}]^{\frac{1}{2}}}\mathbb{W}_m.\]
\item Expression of $\partial_{\lambda \lambda}F_2(\lambda_m^{\pm},0)$.
 \[\partial_{\lambda \lambda} F_2(\lambda_m^{\pm},0)=\frac{4m^2b^{1-m} (m\lambda_m^{\pm}-m+1)^3}{[(m\lambda_m^{\pm}-m+1)^2+b^{2m}]^{\frac{3}{2}}} \mathbb{W}_m . \]
\item Expression of $\partial_{tt} F_2(\lambda_m^{\pm},0)$.
 \begin{align*}
 \partial_{tt} F_2(\lambda_m^{\pm},0)&=-m(m-1)b^{3-3m} \frac{(b^{2m-2}-(m\lambda_m^{\pm}-m+1)^2)^2}{([m\lambda_m^{\pm}-m+1]^2+b^{2m})^{\frac{1}{2}}} \mathbb{W}_m\\
 &+\widetilde{\beta}_m \mathcal{K}_m\mathbb{W}_m
\end{align*}
with 
\begin{eqnarray*}\mathcal{K}_m &\triangleq&\frac{b^{1-m} (m\lambda_m^{\pm}-1)(m\lambda_m^{\pm}-m+1)^2+(1-2m)(m\lambda_m^{\pm}-m+1)b^{m+1}+mb^{3m-1}}{[(m\lambda_m^{\pm}-m+1)^2+b^{2m}]^{\frac{1}{2}}} \\
&\times& (2\lambda_m^{\pm}m-2m+1)
\end{eqnarray*}
and
\[ \widetilde{\beta}_m=-\frac{ 2bm\big( b^{m}-b^{2-m}(m\lambda_m^{\pm}-m+1)\big)^2}{\textnormal{det}(M_{2m}(\lambda_m^{\pm}))}. \]\item Expression of $\partial_{ \lambda}\partial_t F_2(\lambda_m^{\pm},0)$.
\[\partial_{ \lambda}\partial_t F_2(\lambda_m^{\pm},0)=0.\]
\end{enumerate}
\end{proposition}
\begin{remark}\label{rmq7}
In \cite{5}, all the preceding quantities were computed    in the limit case $b=b^*_m$ and our expressions lead to the same thing when we take $b\rightarrow b^*_m$. This can be checked using the identity
$$
m\lambda_m^{\pm}-m+1=-\sqrt{b^{2m}+\Delta_m}\pm\sqrt{\Delta_m}
$$ 
and when $b=b_m^*$ the discriminant $\Delta_m$ vanishes.
\end{remark}
In what follows we shall establish the formulae of Proposition \ref{explicit}. As we can observe from Proposition \ref{ImplicitSS} that most of them are based on   the quantities $\frac{d^k}{dt^k} [G(\lambda_m^{\pm},tv_m)]_{|t=0}$ for $k\in  \lbrace 2,3\rbrace$. We introduce some notations which will be very useful to obtain explicit expressions.
We begin with:
\[\Phi_j(t,w)=b^{j-1} w+tv_{j,m} \overline{w}^{m-1}\]
wich leads to
\[G_j(\lambda_m^{\pm},tv_m)=\textnormal{Im} \left\lbrace [(1-\lambda_m^{\pm}) \overline{\Phi_j(t,w)}+I(\Phi_j(t,w))]w \left(b^{j-1}+t(1-m)v_{j,m}\overline{w}^m\right)\right\rbrace\]
with:
\[ I(\Phi_j(t,w))=I_1(\Phi_j(t,w))-I_2(\Phi_j(t,w))\]
where:
\[I_i(\Phi_j(t,w))=\fint_{\T} \frac{\overline{\Phi_j(t,w)}-\overline{\Phi_i(t,\tau)}}{\Phi_j(t,w)-\Phi_i(t,\tau)} \Phi_i'(t,\tau) d\tau.\]
\subsubsection{Computation of $\partial_t F_2( \lambda_m^{\pm},0)$}
We shall sketch the proof because most of the computations were done in \cite{5}. 
Note that
\[ \partial_t F_2( \lambda_m^{\pm},0)= \frac{1}{2} Q \partial_{ff} G (\lambda_m^{\pm},0)[v_m,v_m]=\frac{1}{2} Q \frac{d^2}{dt^2}[G_j(\lambda_m^{\pm},tv_m)]_{|t=0}.\]
To lighten the notations we introduce  
\[I_i(\Phi_j(t,w))=\fint_{\T} \frac{\overline{A}+t\overline{B}}{A+tB} (b^{i-1}+tC)d\tau\]
with 
\[A=b^{j-1}w-b^{i-1}\tau \text{, }B=v_{j,m}\overline{w}^{m-1}-v_{i,m}\overline{\tau}^{m-1}\quad \text{ and  }\quad C=v_{i,m}(1-m)\overline{\tau}^m.\]
We can easily find that
\begin{align*}
 \frac{d^2}{dt^2}[G_j(\lambda_m^{\pm},tv_m)]_{|t=0} &=\textnormal{Im}\left\lbrace b^{j-1} w\frac{d^2}{dt^2}I(\Phi_j(t,w))_{|t=0}+2(1-\lambda_m^{\pm})(1-m)v_{j,m}^2 \right.\\
&\left. +2(1-m)v_{j,m}  \frac{d}{dt}  I(\Phi_j(t,w))_{|t=0} \overline{w}^{m-1} \right\rbrace\\
&=\textnormal{Im}\left\lbrace b^{j-1}w \frac{d^2}{dt^2}I(\Phi_j(t,w))_{|t=0}+2(1-m)v_{j,m}  \frac{d}{dt} I(\Phi_j(t,w))_{|t=0} \overline{w}^{m-1} \right\rbrace.
\end{align*}
Recall from \cite[p. 823]{5} that
\[  \frac{d}{dt}  [I_i(\Phi_j(t,w))]_{|t=0}=\fint \frac{\overline{A}}{A^2} (AC-b^{i-1}B)d\tau +b^{i-1}\fint \frac{\overline{B}}{A} d\tau.\]
Moreover, for any $ i,j\in \lbrace 1,2 \rbrace$, there  exist real numbers $ \mu_{i,j},\gamma_{i,j}$  such that 
\[\fint \frac{\overline{B}}{A} d\tau=\mu_{i,j} w^{m-1}\]
and

\[\fint \frac{\overline{A}}{A^2} (AC-b^{i-1}B)d\tau=\gamma_{i,j}\overline{w}^{m+1}.\]

Hence,
\[ \frac{d}{dt}[ I_i(\Phi_j(t,w))]_{|t=0}=\gamma_{i,j}\overline{w}^{m+1}+b^{i-1}\mu_{i,j} w^{m-1}\]
with 
\[\gamma_{i,j} \triangleq v_{i,m}(1-m)\fint \frac{b^{j-1}-b^{i-1}\overline{\tau}}{b^{j-1}-b^{i-1}\tau}\overline{\tau}^m d\tau-b^{i-1}\fint \frac{b^{j-1}-b^{i-1}\overline{\tau}}{(b^{j-1}-b^{i-1}\tau)^2}(v_{j,m}-v_{i,m}\overline{\tau}^{m-1})d\tau.\]
We also get from $(39)$ of  \cite{5}, 
\[\gamma_{1,2}=0.\]
For $\gamma_{2,1}$ by writing 
\[\gamma_{2,1}=v_{2,m}(1-m)\fint \frac{1-b\overline{\tau}}{1-b\tau}\overline{\tau}^m d\tau-b\fint \frac{1-b\overline{\tau}}{(1-b\tau)^2}(v_{1,m}-v_{2,m}\overline{\tau}^{m-1})d\tau. \]
combined with the following identities:  for any $ m \in \N^*$
\begin{equation}\label{puiss1}
\fint_{\T} \frac{\overline{\tau}^m}{(1-b\tau)} d\tau=\fint_{\T} \frac{\tau^{m-1}}{\tau-b}d \tau=b^{m-1}
\end{equation}
and 
\begin{equation}\label{puiss2}
\fint_{\T} \frac{\overline{\tau}^m}{(1-b\tau)^2} d\tau=\fint_{\T} \frac{\tau^{m}}{(\tau-b)^2}d \tau=mb^{m-1}.
\end{equation}
We obtain
\begin{align*}
\gamma_{2,1}&=v_{2,m}(1-m)[b^{m-1}-b^{m+1}]+bv_{2,m}[(m-1)b^{m-2}-mb^m]+b^2 v_{1,m}\\
&=-v_{2,m}b^{m+1}+b^2 v_{1,m}.
\end{align*}
For $\gamma_{i,i}$ we recall that 
\[\gamma_{i,i}=v_{i,m}(1-m)\fint \frac{1-\overline{\tau}}{1-\tau}\overline{\tau}^m d\tau-v_{i,m}\fint \frac{1-\overline{\tau}}{(1-\tau)^2}(1-\overline{\tau}^{m-1})d\tau.\]
Thus using the residue theorem at $\infty$ we deduce,
\[ \gamma_{i,i}=0.\]
Finally we get,

%

 \begin{equation}\label{IdQ1}
 \frac{d}{dt} [ I(\Phi_1(t,w))]_{|t=0}=(\mu_{1,1}-b\mu_{2,1})w^{m-1}+[v_{2,m}b^{m+1}-b^2 v_{1,m}]\overline{w}^{m+1}
 \end{equation}
 and
 \begin{equation}\label{IdQ2}
  \frac{d}{dt} [I(\Phi_2(t,w))]_{|t=0} =(\mu_{1,2}-b\mu_{2,2})w^{m-1}.
\end{equation}
    
Now we have to compute $ \frac{d^2}{dt^2}[ I_i(\Phi_j(t,w))]_{|t=0}$.
According to \cite[p. 825-826]{5} one has 
\[ \frac{d^2}{dt^2}[ I_i(\Phi_j(t,w))]_{|t=0}=2\fint_{\T}\frac{[A\overline{B}-\overline{A}B]}{A^3}[AC-b^{i-1}B]d\tau.\]
Moreover 

\begin{align*}
\fint_{\T} \frac{ \overline{B} }{ A^2 }[AC-b^{i-1} B ]d \tau&=\widehat{\mu}_{i,j}\overline{w}\\
\intertext{ and} 
 -\fint_{\T}\frac{\overline{A}B}{A^3}[AC-b^{i-1}B]d\tau &=\eta_{i,j} \overline{w}^{2m+1}.
\end{align*}
 Then we have
 \[\frac{d^2}{dt^2} [I_i(\Phi_j(t,w))]_{|t=0}=2\widehat{\mu}_{i,j}\overline{w}+2\eta_{i,j} \overline{w}^{2m+1}\]
  with 
 \[ \widehat{\mu}_{i,j}=v_{i,m} (1-m)\fint_{\T} \frac{(v_{j,m}-v_{i,m}\tau^{m-1} ) }{ (b^{j-1}-b^{i-1}   \tau )} \overline{\tau}^m d\tau +\fint_{\T} \frac{(v_{j,m}-v_{i,m}\tau^{m-1})}{(b^{j-1} -b^{i-1}  \tau )^2 }b^{i-1}[v_{i,m}\overline{\tau}^{m-1}-v_{j,m}] d\tau \]
 and
 \begin{eqnarray*}\eta_{i,j}&=&\fint_{\T}\frac{(b^{j-1}  -b^{i-1} \overline{\tau} )(v_{j,m}-v_{i,m} \overline{\tau}^{m-1})}{(b^{j-1} -b^{i-1}\tau )^2}v_{i,m}(m-1)\overline{\tau}^m d \tau\\
 &+&b^{i-1}\fint_{\T}\frac{(b^{j-1}  -b^{i-1} \overline{\tau})(v_{j,m} -v_{i,m}\overline{\tau}^{m-1})^2}{(b^{j-1}-b^{i-1}\tau)^3}d\tau.
 \end{eqnarray*}
For the diagonal terms  we get
\begin{eqnarray*}
\widehat{\mu}_{i,i}&=&\frac{v_{i,m}^2}{b^{i-1}} (1-m)\fint_{\T} \frac{(1- \tau^{m-1} ) }{ (1-   \tau )} \overline{\tau}^m d\tau +\frac{v_{i,m}^2}{b^{i-1}}\fint_{\T} \frac{(1-\tau^{m-1})}{(1 -  \tau )^2 }[\overline{\tau}^{m-1}-1] d\tau\\
&=&(m-1)\frac{v_{i,m}^2}{b^{i-1}}
\end{eqnarray*}
As to the term  $\widehat{\mu}_{1,2},$ we may write
 \begin{eqnarray*} \widehat{\mu}_{1,2}&=&v_{1,m} (1-m)\fint_{\T} \frac{(v_{2,m}-v_{1,m}\tau^{m-1} ) }{ (b- \tau )} \overline{\tau}^m d\tau +v_{1,m}\fint_{\T} \frac{(v_{2,m}-v_{1,m}\tau^{m-1})}{(b - \tau )^2 }\overline{\tau}^{m-1} d \tau\\
 & -&v_{2,m} \fint_{\T} \frac{(v_{2,m}-v_{1,m}\tau^{m-1})}{(b - \tau )^2 } d\tau. 
 \end{eqnarray*}
 The first and the second integrals vanish using the residue theorem at $\infty$.  Thus we find   \[ \widehat{\mu}_{1,2}=(m-1) v_{2,m} v_{1,m} b^{m-2}. \]
Concerning the term $\widehat{\mu}_{2,1}$ given by 
 \[ \widehat{\mu}_{2,1}=v_{2,m} (1-m)\fint_{\T} \frac{(v_{1,m}-v_{2,m}\tau^{m-1} ) }{ (1-b \tau )} \overline{\tau}^m d\tau +\fint_{\T} \frac{(v_{1,m}-v_{2,m}\tau^{m-1})}{(1 -b  \tau )^2 }b[v_{2,m}\overline{\tau}^{m-1}-v_{1,m}] d\tau \]
it can be computed using  $(\ref{puiss1})$ and $( \ref{puiss2})$ 
  \begin{eqnarray*} \widehat{\mu}_{2,1}&=&v_{2,m} (1-m)[v_{1,m} b^{m-1}-v_{2,m}]+bv_{2,m}v_{1,m}(m-1)b^{m-2}\\
  &+&b\fint_{\T} \frac{(-v_{2,m}^2-v_{1,m}^2+v_{2,m}v_{1,m}\tau^{m-1} )}{(1 -b  \tau )^2 } d\tau. 
  \end{eqnarray*}
  The last term vanishes thanks to the residue theorem. Finally we have:
   \[ \widehat{\mu}_{2,1}=v_{2,m}^2 (m-1).\]
   
   Now we shall move to the calculation of  $\eta_{i,j}$ for $i,j\in \lbrace 1,2 \rbrace$.
   We start with the term,
    \[\eta_{1,2}=v_{1,m}(m-1)\fint_{\T}\frac{(b - \overline{\tau} )(v_{2,m}-v_{1,m} \overline{\tau}^{m-1})}{(b -\tau )^2}\overline{\tau}^m d \tau +\fint_{\T}\frac{(b  -\overline{\tau})(v_{2,m} -v_{1,m}\overline{\tau}^{m-1})^2}{(b-\tau)^3}d\tau.\]
 Using the residue theorem at $\infty$ we get
       \[\eta_{1,2}=0. \]
 Now we focus on  the term $\eta_{2,1}$ given by 
      
 \[\eta_{2,1}=v_{2,m}(m-1)\fint_{\T}\frac{(1 -b \overline{\tau} )(v_{1,m}-v_{2,m} \overline{\tau}^{m-1})}{(1 -b\tau )^2}\overline{\tau}^m d \tau +b \fint_{\T}\frac{(1 -b \overline{\tau})(v_{1,m} -v_{2,m}\overline{\tau}^{m-1})^2}{(1-b\tau)^3}d\tau.\]
According to   $(\ref{puiss1})$ and $( \ref{puiss2})$ we get
 \begin{eqnarray*}\fint_{\T}\frac{(1 -b \overline{\tau} )(v_{1,m}-v_{2,m} \overline{\tau}^{m-1})}{(1 -b\tau )^2}\overline{\tau}^m d \tau & =&v_{1,m}\Big(mb^{m-1}-(m+1) b^{m+1}\Big)\\
 &+&v_{2,m}\Big(2mb^{2m}-(2m-1)b^{2m-2}\Big).
 \end{eqnarray*}
 Applying the residue theorem,  we can easily prove for any $m \in \N^*$,
 \begin{equation}\label{puiss3}
  \fint_{\T} \frac{\overline{\tau}^m}{(1-b\tau)^3}d\tau= \fint_{\T} \frac{\tau^{m+1}}{(\tau-b)^3} d\tau=\frac{m(m+1)}{2}b^{m-1}.
  \end{equation}
 Thus,
 \begin{align*}
 \fint_{\T}\frac{(1 -b \overline{\tau})(v_{1,m} -v_{2,m}\overline{\tau}^{m-1})^2}{(1-b\tau)^3}d\tau&=v_{1,m}v_{2,m} m\Big((m+1)b^{m}-(m-1)b^{m-2}\Big)-bv_{1,m}^2 \\
 &+v_{2,m}^2(2m-1)\Big((m-1)b^{2m-3}-m b^{2m-1} \Big).
 \end{align*}
It follows  that
  \[\eta_{2,1}=-v_{2,m}^2mb^{2m}-b^2v_{1,m}^2+ v_{1,m}v_{2,m}(m+1)b^{m+1}.\]
 For the diagonal term we write
  \[\eta_{i,i}=\frac{v_{i,m}^2}{b^{i-1} }(m-1)\fint_{\T}\frac{(1  -\overline{\tau} )(1-\overline{\tau}^{m-1})}{(1 -\tau )^2}\overline{\tau}^m d \tau +\frac{v_{i,m}^2}{b^{i-1} } \fint_{\T}\frac{(1  - \overline{\tau})(1 -\overline{\tau}^{m-1})^2}{(1-\tau)^3}d\tau.\]
By the residue theorem we get 
  \[ \eta_{i,i}=0.\]
  Putting together the preceding estimates yields
 \begin{eqnarray*}
  \frac{d^2}{dt^2} [ I(\Phi_1(t,w))]_{|t=0}
  &=&2\Big(v_{2,m}^2mb^{2m}+b^2v_{1,m}^2-v_{1,m} v_{2,m}(m+1)b^{m+1}\Big) \overline{w}^{2m+1} \\
  &+&2(m-1)\big(v_{1,m}^2-v_{2,m}^2 \big) \overline{w}
  \end{eqnarray*}
and
$$
  \frac{d^2}{dt^2} [ I(\Phi_2(t,w))]_{|t=0}=2(m-1)v_{2,m}\Big( v_{1,m} b^{m-2}-\frac{v_{2,m}}{b}\Big)\overline{w}.
$$
 Combining these estimates with \eqref{IdQ1} and \eqref{IdQ2} we find successively,
  \begin{align*}
 \frac{d^2}{dt^2}[G_1(\lambda_m^{\pm},tv_m)]_{|t=0} &=\textnormal{Im} \left\lbrace w \frac{d^2}{dt^2} I(\Phi_1(t,w))_{|t=0}+2(1-m)v_{1,m}  \frac{d}{dt}  I(\Phi_1(t,w))_{|t=0} \overline{w}^{m-1} \right\rbrace\\
  &=2m\big(v_{2,m} b^{m}-b v_{1,m}\big)^2 e_{2m}
\end{align*}
and
\begin{align*}
  \frac{d^2}{dt^2}[G_2(\lambda_m^{\pm},tv_m)]_{|t=0} &=\textnormal{Im}\left\lbrace b w \frac{d^2}{dt^2}I(\Phi_2(t,w))_{|t=0}+2(1-m)v_{2,m}  \frac{d}{dt}  I(\Phi_2(t,w))_{|t=0} \overline{w}^{m-1} \right\rbrace\\
&=0.
\end{align*}
This can be written in the form,
\begin{align*}
 \frac{d^2}{dt^2}[G(\lambda_m^{\pm},tv_m)]_{|t=0}= \left(\begin{array}{c}
  2m(v_{2,m} b^{m}-b v_{1,m})^2 \\ 
  0
  \end{array}  \right) e_{2m}.
  \end{align*}
From the structure of the projector $Q$ we get
\begin{align*}
\partial_t F_2( \lambda_m^{\pm},0) &= \frac{1}{2} Q \partial_{ff} G (\lambda_m^{\pm},0)[v_m,v_m]\\
&=\frac{1}{2} Q \frac{d^2}{dt^2} [G(\lambda_m^{\pm},tv_m)]_{|t=0}\\
&=0.
\end{align*}
Hence the first point of the Proposition $\ref{explicit}$ is proved.
\subsubsection{Computation of  $\partial_\lambda F_2(\lambda_m^{\pm},0)$}
From the explicit expression of  $\partial_f G(\lambda_m^{\pm},0)$ it is easy to verify that
\begin{equation}\label{crossvm}
\partial_\lambda \partial_f G(\lambda_m^{\pm},0)(v_m)=m\left(\begin{array}{c}
v_{1,m} \\ 
bv_{2,m}
\end{array}  \right) e_m.
\end{equation}

Thus we have,
\begin{align*}
\partial_\lambda F_2(\lambda_m^{\pm},0)&= Q \partial_\lambda \partial_f G(\lambda_m^{\pm},0)(v_m)\\
&=m\Big\langle \left(\begin{array}{c}
v_{1,m} \\ 
bv_{2,m}
\end{array}  \right), \widehat{\mathbb{W}}_m \Big\rangle\,\, \mathbb{W}_m.
\end{align*}
Straightforward computations lead to
\begin{equation}\label{F_2lambda}
\partial_\lambda F_2(\lambda_m^{\pm},0)=\frac{m[(m\lambda_m^{\pm}-m+1)^2-b^{2m}]}{b^{m-1}[(m\lambda_m^{\pm}-m+1)^2+b^{2m}]} \left(\begin{array}{c}
m\lambda_m^{\pm}-m+1 \\ 
-b^m
\end{array}  \right) e_{m}.
\end{equation}
Thus the second point of the Proposition $\ref{explicit}$ follows.
  \subsubsection{Computation of $\partial_{\lambda \lambda}F_2(\lambda_m^{\pm},0)$}
 Using  $(\ref{crossvm})$ and$(\ref{F_2lambda})$ we obtain
\begin{align*}
(\hbox{Id}-Q)\partial_{\lambda}\partial_fG(\lambda_m^{\pm},0)v_m&=\frac{2mb(m\lambda_m^{\pm}-m+1)}{(m\lambda_m^{\pm}-m+1)^2+b^{2m}}\left(\begin{array}{c}
b^m \\ 
m\lambda_m^{\pm}-m+1
\end{array}  \right) e_m\\
& = \kappa\left(\begin{array}{c}
b^m \\ 
m\lambda_m^{\pm}-m+1
\end{array}  \right) e_m
\end{align*}
with
\[ \kappa \triangleq \frac{2mb(m\lambda_m^{\pm}-m+1)}{(m\lambda_m^{\pm}-m+1)^2+b^{2m}}\cdot\]
Then by \eqref{Invert11} and the expression of $\partial_\lambda\partial_g\varphi(\lambda_m^\pm,0)$ detailed in Proposition \ref{ImplicitSS}  one gets
\begin{align}\label{phiZZ}
\nonumber\partial_\lambda\partial_g\varphi(\lambda_m^\pm,0)&=-[\partial_fG(\lambda_m^{\pm},0)]^{-1}(\hbox{Id}-Q)\partial_{\lambda}\partial_fG(\lambda_m^{\pm},0)v_m\\
\nonumber&=-\kappa[\partial_fG(\lambda_m^{\pm},0)]^{-1}\left(\begin{array}{c}
b^m \\ 
m\lambda_m^{\pm}-m+1
\end{array}  \right) e_m\\
&=\frac{2mb^{1-m} (m\lambda_m^{\pm}-m+1)^2}{\big(m\lambda_m^{\pm}-m+1\big)^2+b^{2m}}\left( \begin{array}{c}
1 \\ 
0
\end{array} \right) \overline{w}^{m-1}.
\end{align}
%
%
Consequently,
\[ \partial_{\lambda} \partial_f G(\lambda_m^{\pm},0)[ \partial_{\lambda} \partial_g \varphi(\lambda_m^{\pm},0)v_m]= \frac{2m^2(m\lambda_m^{\pm}-m+1)^2}{(m\lambda_m^{\pm}-m+1)^2+b^{2m}} \left( \begin{array}{c}
1 \\ 
 0
\end{array} \right)e_m.\]
Straightforward computations lead to
\[Q\partial_{\lambda} \partial_f G(\lambda_m^{\pm},0)[ \partial_{\lambda} \partial_g \varphi(\lambda_m^{\pm},0)v_m]=\frac{2m^2 b^{1-m}(m\lambda_m^{\pm}-m+1)^3}{[(m\lambda_m^{\pm}-m+1)^2+b^{2m}]^2} \left( \begin{array}{c}
m\lambda_m^{\pm}-m+1 \\ 
-b^m
\end{array} \right)e_m.\]
Finally we obtain the following expression
\[\partial_{\lambda \lambda} F_2(\lambda_m^{\pm},0)=\frac{4m^2b^{1-m} (m\lambda_m^{\pm}-m+1)^3}{[(m\lambda_m^{\pm}-m+1)^2+b^{2m}]^{\frac{3}{2}}} \mathbb{W}_m. \]
  \subsubsection{Computation of $\partial_{tt} F_2(\lambda_m^{\pm},0)$}
  We mention that most of the computations were done in \cite{5} and so we shall just outline the basic steps. 
  Looking to the formula given in \mbox{Proposition \ref{ImplicitSS}} we need first to  compute $\frac{d^3}{dt^3}  [G(\lambda_m^{\pm},tv_m)]_{|t=0}$. From the identity $(60)$ of \cite{5} we recall that 

\[ \frac{d^3}{dt^3} G_j(\lambda_m^{\pm},tv_m)_{|t=0} = \textnormal{Im} \left\lbrace b^{j-1} w  \frac{d^3}{dt^3}[I(\Phi_j(t,w))]_{|t=0}+3(1-m)v_{j,m}\overline{w}^{m-1} [ \frac{d^2}{dt^2}(I(\Phi_j(t,w)))]_{|t=0} \right\rbrace .\]

It is also shown in \cite[p. 835]{5} that

\[  \frac{d^3}{dt^3} [I_i(\Phi_j(t,w))]_{|t=0}=-6 \fint_{\T} \frac{[\overline{B}A-\overline{A}B]}{A^4}B[AC-b^{i-1}B]d \tau. \]
One can find real numbers $ \widehat{\eta}_{i,j} $  such that 

\[\frac{1}{6}\frac{d^3}{dt^3}[I_i(\Phi_j(t,w))]_{|t=0}=(m-1)J_{i,j}\overline{w}^{m+1}+b^{i-1}K_{i,j}\overline{w}^{m+1}+\widehat{\eta}_{i,j} \overline{w}^{3m+1}\]
where

\[J_{i,j}= v_{i,m}\fint_{\T} \frac{(v_{j,m} -v_{i,m}\overline{\tau}^{m-1})(v_{j,m} -v_{i,m} \tau^{m-1})}{(b^{j-1}-b^{i-1}\tau)^2}\overline{\tau}^m d\tau\]
and
\[K_{i,j}=\fint_{\T} \frac{(v_{j,m} -v_{i,m}\overline{\tau}^{m-1})^2 (v_{j,m} -v_{i,m} \tau^{m-1})}{(b^{j-1} -b^{i-1}\tau)^3} d\tau. \]

To start, we compute $J_{i,j}$. The same proof of $(61)$ of \cite{5} gives 
\[ J_{1,2}=0.\]
For $J_{2,1}$ we use $(\ref{puiss2})$,
\begin{align*}
J_{2,1}&= v_{2,m}\fint_{\T} \frac{(v_{1,m} -v_{2,m}\overline{\tau}^{m-1})(v_{1,m} -v_{2,m} \tau^{m-1})}{(1-b \tau)^2}\overline{\tau}^m d\tau\\
&=v_{2,m}[v_{1,m}^2+v_{2,m}^2] mb^{m-1}+v_{1,m}v_{2,m}^2[(1-2m)b^{2m-2}-1].\\
\end{align*}

As to  $J_{i,i}$ we easily get
\begin{eqnarray*}J_{i,i} &=& \frac{v_{i,m}^3}{b^{2(i-1)}}\fint_{\T} \frac{(1 -\overline{\tau}^{m-1})(1 - \tau^{m-1})}{(1-\tau)^2}\overline{\tau}^m d\tau\\
&=&0.
\end{eqnarray*}

For $K_{1,2}$ we write 

\begin{eqnarray*}K_{1,2}&=&\fint_{\T} \frac{v_{2,m}^3 +v_{1,m}^2v_{2,m} \overline{\tau}^{2m-2}-2v_{2,m}^2 v_{1,m}\overline{\tau}^{m-1}-v_{1,m}^3\overline{\tau}^{m-1}+2v_{1,m}^2v_{2,m}}{(b- \tau)^3} d\tau\\
&-&v_{1,m}v_{2,m}^2\fint_{\T} \frac{\tau^{m-1}}{(b-\tau)^3}d\tau.
\end{eqnarray*}

Using the residue theorem at $\infty$, we can see that all the terms vanish except the last one that can computed also with  the residue theorem.
\[K_{1,2}=\frac{v_{1,m}v_{2,m}^2(m-1)(m-2)}{2}b^{m-3}. \]
For $K_{2,1}$ given by 

\[K_{2,1}=\fint_{\T} \frac{v_{1,m}^3 -v_{1,m}^2 v_{2,m} \tau^{m-1}+2 v_{1,m} v_{2,m}^2 +v_{1,m} v_{2,m}^2 \overline{\tau}^{2m-2}-v_{2,m}^3 \overline{\tau}^{m-1}-2 v_{1,m}^2 v_{2,m}\overline{\tau}^{m-1}}{(1 -b\tau)^3} d\tau.\]

we may use  the residue  theorem combined  with $(\ref{puiss3})$
\begin{align*}
K_{2,1}&=\fint_{\T} \frac{v_{1,m} v_{2,m}^2 \overline{\tau}^{2m-2}-v_{2,m}^3 \overline{\tau}^{m-1}-2 v_{1,m}^2 v_{2,m}\overline{\tau}^{m-1}}{(1 -b\tau)^3} d\tau\\
&=(2v_{1,m}^2+v_{2,m}^2)v_{2,m}\frac{m(1-m)}{2}b^{m-2}+v_{1,m} v_{2,m}^2(m-1)(2m-1)b^{2m-3}.
\end{align*}
As to the diagonal terms $K_{i,i}$ we use residue theorem leading to
\begin{eqnarray*}K_{i,i}&=&\frac{v_{i,m}^3}{b^{3(i-1)}}\fint_{\T} \frac{( 1- \overline{\tau}^{m-1})^2 (1 - \tau^{m-1})}{(1-\tau)^3} d\tau\\
&=&\frac{v_{i,m}^3}{b^{3(i-1)}}\frac{(m-1)(m-2)}{2}
\end{eqnarray*}

Summing up we find
\begin{eqnarray*}
\frac{1}{6}\frac{d^3}{dt^3}[I(\Phi_1(t,w))]_{|t=0} &=&(m-1)\Big(v_{2,m}^2 v_{1,m}-v_{2,m}^3 \frac{m}{2}b^{m-1}+v_{1,m}^3\frac{m-2}{2}\Big)\overline{w}^{m+1}\\
&+&(\widehat{\eta}_{1,1}- \widehat{\eta}_{2,1})\overline{w}^{3m+1}
\end{eqnarray*}
{ and}
\begin{eqnarray*}
\frac{1}{6}\frac{d^3}{dt^3}  [I(\Phi_2(t,w))]_{|t=0}&=&\frac{v_{2,m}^2(m-1)(m-2)}{2}\Big(v_{1,m}b^{m-3}-\frac{v_{2,m}}{b^{2}}\Big) \overline{w}^{m+1}\\
&+&(\widehat{\eta}_{1,2}- \widehat{\eta}_{2,2})\overline{w}^{3m+1}.
\end{eqnarray*}
This leads to
\begin{eqnarray*}
\frac{d^3}{dt^3} [G_1(\lambda_m^{\pm},tv_m)]_{|t=0}&=& \textnormal{Im} \Big\{  6(m-1)\Big(v_{2,m}^2 v_{1,m}-v_{2,m}^3 \frac{m}{2}b^{m-1}+v_{1,m}^3\frac{m-2}{2}\Big)\overline{w}^{m} .\\
& +&6(\widehat{\eta}_{1,1}- \widehat{\eta}_{2,1})\overline{w}^{3m}+6(1-m)v_{1,m} (m-1)\Big(v_{1,m}^2-v_{2,m}^2 \Big)\overline{w}^{m}\\
& +&6(1-m)v_{1,m}\Big(v_{2,m}^2mb^{2m}+b^2v_{1,m}^2-v_{1,m} v_{2,m}(m+1)b^{m+1}\Big)\overline{w}^{3m} \Big\} \\
&=&3m(m-1)\Big(2v_{2,m}^2 v_{1,m}-v_{2,m}^3 b^{m-1}-v_{1,m}^3\Big)e_m+\gamma_1e_{3m}
\end{eqnarray*}
and 
\begin{align*}
 \frac{d^3}{dt^3} [G_2(\lambda_m^{\pm},0) ]_{|t=0} &= \textnormal{Im} \left\lbrace 3v_{2,m}^2(m-1)(m-2)\Big(v_{1,m}b^{m-2}-\frac{v_{2,m}}{b}\Big) \overline{w}^{m}+b(\widehat{\eta}_{1,2}- \widehat{\eta}_{2,2})\overline{w}^{3m} \right.\\
 &\left.+6(1-m)v_{2,m}^2 (m-1)\Big(v_{1,m} b^{m-2}-\frac{v_{2,m}}{b}\Big)\overline{w}^{m} \right\rbrace \\
 &=3v_{2,m}^2(m-1)m\Big(\frac{v_{2,m}}{b}-v_{1,m}b^{m-2}\Big)e_m+\gamma_2 e_{3m}
 \end{align*}
 with $\gamma_j\in\R.$ 
In summary,
\[\frac{d^3}{dt^3} [G(\lambda_m^{\pm},0) ]_{|t=0} =3m(m-1) \left(\begin{array}{c}
2v_{2,m}^2 v_{1,m}-v_{2,m}^3 b^{m-1}-v_{1,m}^3 \\ 
v_{2,m}^2\Big(\frac{v_{2,m}}{b}-v_{1,m}b^{m-2}\Big)
\end{array}  \right) e_m+\left( \begin{array}{c}
\gamma_1 \\ 
\gamma_2
\end{array} \right)e_{3m}.\]

Using the structure of the projector $Q$ we deduce after algebraic cancellations 
%
\[\frac{1}{3} \frac{d^3}{dt^3} [QG(\lambda_m^{\pm},tv_m)]_{| t=0}=-m(m-1)b^{3-3m} \frac{(b^{2m-2}-(m\lambda_m^{\pm}-m+1)^2)^2}{([m\lambda_m^{\pm}-m+1]^2+b^{2m})^{\frac{1}{2}}} \mathbb{W}_m.\]
Now we shall  compute the term,
\[ Q \partial_{ff}G(\lambda_m^{\pm},0)[v_m, \widehat{v}_m]= Q \partial_t\partial_s[ G(\lambda_m^{\pm},tv_m+s\widehat{v}_m)]_{|t=0,s=0}. \]
To find an expression of $\widehat{v}_m$,  we recall that
\[ \frac{d^2}{dt^2} [G(\lambda_m^{\pm},tv_m)]_{|t=0}= \left(\begin{array}{c}
  \widehat{\alpha} \\ 
  0
  \end{array}  \right) e_{2m} \quad \text{  with }\quad  \widehat{\alpha}=2m(v_{2,m} b^{m}-b v_{1,m})^2.\]
Thus,
\begin{align}\label{vm}
\widehat{v}_m&=-[\partial_f G(\lambda_m^{\pm},0)]^{-1} \left(\begin{array}{c}
  \widehat{\alpha} \\ 
  0
  \end{array}  \right) e_{2m}\\
  \nonumber&=-M_{2m}^{-1}\left(\begin{array}{c}
  \widehat{\alpha} \\ 
  0
  \end{array}  \right)  \overline{w}^{2m-1}\\
  \nonumber&= \left(\begin{array}{c}
  \widehat{v}_{1,m} \\ 
  \widehat{v}_{2,m}
  \end{array}  \right) \overline{w}^{2m-1}
  \end{align}
  with 
  \[ \widehat{v}_{1,m}=-\frac{b\widehat{\alpha}( 2m\lambda_m^{\pm}-2m+1)}{\textnormal{det}(M_{2m}(\lambda_m^{\pm}))} \quad \text{ and } \quad\widehat{v}_{2,m}=-\frac{b^{2m} \widehat{\alpha}}{\textnormal{det}(M_{2m}(\lambda_m^{\pm}))}\cdot \]

Finally we can write,
\[\widehat{v}_m=\widetilde{\beta}_m \tilde{v}_m\]
with 
\begin{equation}\label{beta}\widetilde{\beta}_m=-\frac{b\widehat{\alpha}}{\textnormal{det}(M_{2m}(\lambda_m^{\pm}))} 
\end{equation}
\begin{equation*}
 \tilde{v}_m=\left( \begin{array}{c}
2m\lambda_m^{\pm}-2m+1 \\ 
b^{2m-1}
\end{array} \right)\overline{w}^{2m-1}\triangleq \left( \begin{array}{c}
\beta_1 \\ 
\beta_2
\end{array} \right)\overline{w}^{2m-1}.
\end{equation*}
It follows that
\[  Q \partial_t\partial_s[G(\lambda_m^{\pm},tv_m+s\widehat{v}_m)]_{|t=0,s=0}= \tilde{\beta}_mQ \partial_t\partial_s[ G(\lambda_m^{\pm},tv_m+s\tilde{v}_m)]_{|t=0,s=0}.\]
We shall  introduce the functions 
\[\varphi_j(t,s,w)=b^{j-1}w+tv_{j,m}{w}^{1-m}+s\beta_j{w}^{1-2m}\]
and hence
\begin{align*}
G_j(t,s,w)&\triangleq G_j(\lambda,tv_m+s\tilde{v}_m)\\
&=\textnormal{Im} \left\lbrace \left[(1-\lambda_m^{\pm})\overline{\varphi_j(t,s,w)}+I(\varphi_j(t,s,w))\right]w \partial_w\varphi_j(t,s,w)\right \rbrace.
\end{align*}
The following equality can be easily check :
\begin{eqnarray*}
 \partial_t\partial_s[G_j(t,s,w)]_{|t=0,s=0}& =&\textnormal{Im} \Big\{(1-\lambda_m^{\pm})\beta_j v_{j,m}\Big((1-m)w^m+ (1-2m)\overline{w}^m \Big)\Big\}\\
 &+& \textnormal{Im} \Big\{ w b^{j-1} \frac{d^2}{dtds} [I(\varphi_j(t,s,w))]_{|t=0,s=0}\Big\} \\
 &+&\textnormal{Im} \Big\{\beta_j(1-2m)  \partial_t  [I(\varphi_j(t,s,w))]_{|t=0,s=0}\overline{w}^{2m-1} \Big\} \\
&+&\textnormal{Im} \Big\{(1-m)v_{j,m} \overline{w}^{m-1} \partial_s[I(\varphi_j(t,s,w))]_{|t=0,s=0} \Big\}.
\end{eqnarray*}
We write
\begin{align*}
I_i(\varphi_j(t,s,w))&=\fint_{\T}\frac{\overline{\varphi_j(t,s,w)}-\overline{\varphi_i(t,s,\tau)}}{\varphi_j(t,s,w)-\varphi_i(t,s,\tau)}\Big(b^{i-1}+t(1-m)v_{i,m}\overline{\tau}^m+s\beta_i(1-2m)\overline{\tau}^{2m}\Big)d\tau\\
&=\fint_{\T} \frac{\overline{A}+t\overline{B}+s\overline{C}}{A+tB+sC}\big(b^{i-1}+t\,D+s\,E\big)d\tau
\end{align*}
with 
\[A=b^{j-1}w-b^{i-1}\tau \text{, }B=v_{j,m}\overline{w}^{m-1}-v_{i,m}\overline{\tau}^{m-1} \text{, } C=\beta_j\overline{w}^{2m-1}-\beta_i \overline{\tau}^{2m-1}\]
\[D=(1-m)v_{i,m}\overline{\tau}^m\quad \text{ and  } \quad E=(1-2m)\beta_i\overline{\tau}^{2m}.\]
Straightforward computations  lead to
\begin{eqnarray*} \partial_t[I_i(\varphi_j(t,s,w)]_{|t=0,s=0}&=&b^{i-1} \fint_{\T} \frac{A\overline{B}-\overline{A}B}{A^2}d\tau+ \fint_{\T} \frac{\overline{A}D}{A}d\tau\\
&=&\tilde{J}_{i,j} w^{m-1}+ \theta_{i,j} \overline{w}^{m+1}
\end{eqnarray*}
with $\theta_{i,j}\in\R$ and 
\[\tilde{J}_{i,j}= b^{i-1}\fint_{\T} \frac{(v_{j,m} -v_{i,m} \tau^{m-1})}{(b^{j-1} -b^{i-1}\tau)} d\tau. \]

For the diagonal tens one has
\begin{eqnarray*}\tilde{J}_{i,i}&=& v_{i,m}\fint_{\T} \frac{(1 - \tau^{m-1})}{(1 - \tau)} d\tau\\
&=&0
\end{eqnarray*}
On the other hand
\begin{eqnarray*}\tilde{J}_{1,2}&=& \fint_{\T} \frac{(v_{2,m} -v_{1,m} \tau^{m-1})}{(b -\tau)} d\tau\\
&=&v_{1,m}b^{m-1}-v_{2,m}.
\end{eqnarray*}
Using again the residue theorem it is easy to see that,
\[\tilde{J}_{2,1}= b\fint_{\T} \frac{(v_{1,m} -v_{2,m} \tau^{m-1})}{(1 -b\tau)} d\tau=0.\]
Summing up we obtain,
\begin{align*}
\partial_t[I(\varphi_1(t,s,w))]_{|t=0,s=0}&=(\theta_{1,1}-\theta_{2,1})\overline{w}^{m+1}\\
\intertext{ and }
\partial_t[I(\varphi_2(t,s,w))]_{|t=0,s=0}&=(\theta_{1,2}-\theta_{2,2})\overline{w}^{m+1}+(v_{1,m} b^{m-1}-v_{2,m})w^{m-1}.
\end{align*}
For the derivative with respect to $s$ we may write
\begin{eqnarray*}  \partial_s[I_i(\varphi_j(t,s,w))]_{|t=0,s=0}&=&b^{i-1} \fint_{\T} \frac{A\overline{C}-\overline{A}C}{A^2}d\tau+ \fint_{\T} \frac{\overline{A}E}{A}d\tau\\
&=&\widehat{J}_{i,j} w^{2m-1}+ \widehat{\theta}_{i,j} \overline{w}^{2m+1}
\end{eqnarray*}
where
\begin{equation*}
\widehat{J}_{i,j}= b^{i-1}\fint_{\T} \frac{(\beta_j -\beta_i \tau^{2m-1})}{(b^{j-1} -b^{i-1}\tau)} d\tau
\end{equation*}
and
\begin{align*}
\widehat{\theta}_{i,j}&=\fint_{\T} \frac{(b^{j-1} -b^{i-1}\overline{\tau})}{(b^{j-1} -b^{i-1}\tau)}(1-2m)\beta_i\overline{\tau}^{2m} d\tau -b^{i-1} \fint_{\T} \frac{(b^{j-1} -b^{i-1}\overline{\tau})(\beta_j -\beta_i \overline{\tau}^{2m-1})}{(b^{j-1} -b^{i-1}\tau)^2}d\tau.
\end{align*}
It is easy  to check that $\widehat{\theta}_{i,j}\in \R$, $\forall i,j \in \lbrace 1,2 \rbrace$.
Now, we get 
\[\widehat{J}_{i,i}= \beta_i \fint_{\T} \frac{(1 -  \tau^{2m-1})}{(1 - \tau)} d\tau=0.\]
Using  the residue theorem we find
\[\widehat{J}_{1,2}= \fint_{\T} \frac{(\beta_2-\beta_1 \tau^{2m-1})}{(b -\tau)} d\tau=-\beta_2+\beta_1 b^{2m-1}.\]
and
\[\widehat{J}_{2,1}= b\fint_{\T} \frac{(\beta_1 -\beta_2 \tau^{2m-1})}{(1 -b\tau)} d\tau=0.\]
To summarize,
\begin{align*} 
 \partial_s[I(\varphi_1(t,s,w))]_{|t=0,s=0}&=( \widehat{\theta}_{1,1}-\widehat{\theta}_{2,1})\overline{w}^{2m+1} \\
 \intertext{ and}
\partial_s[I(\varphi_2(t,s,w))]_{|t=0,s=0}&=(\beta_1 b^{2m-1}-\beta_2)w^{2m-1}+(\widehat{\theta}_{1,2}-\widehat{\theta}_{2,2})\overline{w}^{2m+1}.
 \end{align*}
 Now we shall move to  the second derivative with respect to $t$ and $s$,
 \begin{align*}
 \frac{d^2}{dsdt}[I_i(\varphi_j(t,s,w))]_{|t=0,s=0}&=-b^{i-1} \fint_{\T} \frac{\overline{B}C}{A^2}d\tau+ \fint_{\T} \frac{\overline{B}E}{A}d\tau-b^{i-1} \fint_{\T} \frac{\overline{C}B}{A^2}d\tau+\fint_{\T}\frac{\overline{C}D}{A}d\tau\\
  &+2 b^{i-1}\fint_{\T} \frac{BC\overline{A}}{A^3}d\tau- \fint_{\T} \frac{BE\overline{A}}{A^2}d\tau-\fint_{\T} \frac{DC\overline{A}}{A^2}d\tau.
  \end{align*}
  By homogeneity, there exist   $ \varepsilon_{i,j} \in \R $  such that,

\[ \frac{d^2}{dsdt}[I_i(\varphi_j(t,s,w)])_{|t=0,s=0}=\varepsilon_{i,j} \overline{w}^{3m+1} -b^{i-1}I_1^{i,j} w^{m-1} -b^{i-1}I_2^{i,j} \overline{w}^{m+1}+I_3^{i,j}\overline{w}^{m+1}+I_4^{i,j} w^{m-1}\]
with 
\begin{align*}
I_1^{i,j}&=\fint_{\T} \frac{(v_{j,m} -v_{i,m}\overline{\tau}^{m-1})(\beta_j -\beta_i \tau^{2m-1})}{(b^{j-1} -b^{i-1}\tau)^2}d\tau,\\
I_2^{i,j}&=\fint_{\T} \frac{(v_{j,m} -v_{i,m} \tau^{m-1})(\beta_j -\beta_i \overline{\tau}^{2m-1})}{(b^{j-1} -b^{i-1}\tau)^2}d\tau,\\
I_3^{i,j}&=(1-2m)\beta_i \fint_{\T} \frac{(v_{j,m} -v_{i,m} \tau^{m-1})}{(b^{j-1} -b^{i-1}\tau)} \overline{\tau}^{2m} d\tau \\
\intertext{ and}
I_4^{i,j}&=(1-m)v_{i,m} \fint_{\T}\frac{(\beta_j -\beta_i \tau^{2m-1})}{(b^{j-1} -b^{i-1}\tau)}\overline{\tau}^m d\tau.
\end{align*}
We intend to   compute all these terms. For the diagonal terms we write

\begin{eqnarray*}I_1^{i,i}&=& \frac{v_{i,m}\beta_i}{b^{2(i-1)}}\fint_{\T} \frac{(1-\overline{\tau}^{m-1})(1 -\tau^{2m-1})}{(1 - \tau)^2}d\tau\\
&=&(1-m)\frac{v_{i,m} \beta_i}{b^{2(i-1)}}\cdot
\end{eqnarray*}

As to the term $I_1^{1,2}$ we have

\[I_1^{1,2}=\fint_{\T} \frac{v_{2,m}\beta_2 }{(b -\tau)^2}d\tau-\fint_{\T}\frac{v_{1,m}\beta_2\overline{\tau}^{m-1}}{(b -\tau)^2}d\tau-\fint_{\T} \frac{v_{2,m} \beta_1 \tau^{2m-1}}{(b -\tau)^2}d\tau+\fint_{\T} \frac{v_{1,m}\beta_1 \tau^{m})}{(b -\tau)^2}d\tau.\]

By the residue theorem we get  
\[I_1^{1,2}=v_{2,m} \beta_1(1-2m)b^{2m-2}+m\beta_1 v_{1,m} b^{m-1}. \]
Now we move to  $I_1^{2,1}$. Residue theorem combined with $(\ref{puiss2})$ implies

\begin{eqnarray*}I_1^{2,1}&=&\fint_{\T} \frac{v_{1,m} \beta_1 }{(1 -b\tau)^2}d\tau+\fint_{\T} \frac{v_{2,m}\beta_2 \tau^{m}}{(1 -b\tau)^2}d\tau-\fint_{\T} \frac{v_{1,m} \beta_2 \tau^{2m-1}}{(1 -b\tau)^2}d\tau-\fint_{\T} \frac{v_{2,m}\beta_1\overline{\tau}^{m-1} }{(1 -b\tau)^2}d\tau\\
&=&-v_{2,m} \beta_1(m-1)b^{m-2}.
\end{eqnarray*}

Moreover,
\begin{eqnarray*}I_2^{i,i}&=&\frac{v_{i,m} \beta_i}{b^{2(i-1)}} \fint_{\T} \frac{(1 -  \tau^{m-1})(1-  \overline{\tau}^{2m-1})}{(1 - \tau)^2}d\tau\\
&=&(1-m)\frac{v_{i,m}\beta_i}{b^{2(i-1)}}.
\end{eqnarray*}

For $I_2^{i,j}$ we use   the change of variable $\tau \to\overline{\tau}$ 
\begin{align*}
I_2^{i,j}&=\fint_{\T} \frac{(v_{j,m} -v_{i,m} \tau^{m-1})(\beta_j -\beta_i \overline{\tau}^{2m-1})}{(b^{j-1} -b^{i-1}\tau)^2}d\tau\\
&=\fint_{\T} \frac{(v_{j,m} -v_{i,m} \overline{\tau}^{m-1})(\beta_j -\beta_i \tau^{2m-1})}{(b^{i-1} -b^{j-1}\tau)^2}d\tau.
\end{align*}
Therefore
\[I_2^{1,2}=\fint_{\T} \frac{(v_{2,m} -v_{1,m} \overline{\tau}^{m-1})(\beta_2 -\beta_1 \tau^{2m-1})}{(1 -b\tau)^2}d\tau.\]
Similarly to  $I_1^{2,1}$ we find 
\[I_2^{1,2}=-v_{1,m}\beta_2(m-1)b^{m-2}.\]
For the term $I_2^{2,1}$ we write 
\[I_2^{2,1}=\fint_{\T} \frac{(v_{1,m} -v_{2,m} \overline{\tau}^{m-1})(\beta_1 -\beta_2 \tau^{2m-1})}{(b - \tau)^2}d\tau.\]
The same computations  for  $I_1^{1,2}$ yield
\[I_2^{2,1}=v_{1,m}\beta_2(1-2m)b^{2m-2}+m\beta_2v_{2,m} b^{m-1}.\]
For the diagonal term $I_3^{i,i}$ we easily get 

\[I_3^{i,i}=(1-2m)\frac{v_{i,m}\beta_i}{b^{i-1}} \fint_{\T} \frac{(1 - \tau^{m-1})}{(1 -\tau)} \overline{\tau}^{2m} d\tau=0.\]
Moreover,
\[I_3^{1,2}=(1-2m)\beta_i \fint_{\T} \frac{(v_{2,m} -v_{1,m} \tau^{m-1})}{(b - \tau)} \overline{\tau}^{2m} d\tau=0.\]

On the other hand, using $(\ref{puiss1})$ we find
\begin{eqnarray*}I_3^{2,1}&=&(1-2m)\beta_2 \fint_{\T} \frac{(v_{1,m} -v_{2,m} \tau^{m-1})}{(1 -b\tau)} \overline{\tau}^{2m} d\tau\\
&=&
(1-2m)\beta_2 \big(v_{1,m}b^{2m-1}-v_{2,m}b^m\big). 
\end{eqnarray*}
Now we move to the last terms $I_4^{i,j}.$ Concerning the diagonal terms, we may write 

\begin{eqnarray*}I_4^{i,i}&=&(1-m)\frac{v_{i,m} \beta_i}{b^{i-1}} \fint_{\T}\frac{(1 - \tau^{2m-1})}{(1 -\tau)}\overline{\tau}^m d\tau\\
&=&(1-m)\frac{v_{i,m} \beta_i}{b^{i-1}}.
\end{eqnarray*}

For $I_4^{1,2}$ we obtain according to the residue theorem

\begin{eqnarray*}I_4^{1,2}&=& (1-m)v_{1,m} \fint_{\T}\frac{\beta_2 \overline{\tau}^m}{(b -\tau)} d\tau-(1-m)v_{1,m} \fint_{\T}\frac{\beta_1 \tau^{m-1}}{(b -\tau)}d\tau\\
&=&(1-m) v_{1,m} \beta_1 b^{m-1}.
\end{eqnarray*}

For the last term, we use  $(\ref{puiss1})$ in order to get
\begin{eqnarray*}I_4^{2,1}&=&(1-m)v_{2,m} \fint_{\T}\frac{(\beta_1 -\beta_2 \tau^{2m-1})}{(1 -b\tau)}\overline{\tau}^m d\tau\\
&=&(1-m)v_{2,m} \beta_1 b^{m-1}.
\end{eqnarray*}

 Putting together the preceding identities we deduce
  \begin{align*}
 \frac{d^2}{dsdt}[I(\varphi_1(t,s,w))]_{|t=0,s=0}&=(\varepsilon_{1,1}-\varepsilon_{2,1}) \overline{w}^{3m+1}+(m-1)\big(v_{1,m}\beta_1-\beta_2v_{2,m} b^m\big)\overline{w}^{m+1} \\
  \intertext{and }
  \frac{d^2}{dsdt}[I(\varphi_2(t,s,w))]_{|t=0,s=0}&=(\varepsilon_{1,2}-\varepsilon_{2,2}) \overline{w}^{3m+1}+(1-2m)\beta_1 \big(v_{1,m} b^{m-1}-v_{2,m} b^{2m-2} \big)w^{m-1}\\
 & +(m-1)\beta_2\big(v_{1,m}b^{m-2}-\frac{v_{2m}}{b} \big)\overline{w}^{m+1}.
 \end{align*}
Finally we get,
 \begin{eqnarray*}
 \frac{d^2}{dtds} [G_1(t,s,w)]_{|t=0,s=0} &=& \textnormal{Im} \Big\{ (\varepsilon_{1,1}-\varepsilon_{2,1}) \overline{w}^{3m}+(m-1)[v_{1,m}\beta_1-\beta_2v_{2,m} b^m]\overline{w}^{m} \\
 &+&\beta_1(1-2m) (\theta_{1,1}-\theta_{2,1})\overline{w}^{3m} \Big\} \\
& +&\textnormal{Im} \Big\{(\lambda_m^{\pm}-1)m\beta_1 v_{1,m}\overline{w}^m ]+(1-m)v_{1,m}( \widehat{\theta}_{1,1}-\widehat{\theta}_{2,1})\overline{w}^{3m} \Big\}\\
&=&\left((m\lambda_m^{\pm}-1)v_{1,m}\beta_1+(1-m)\beta_2v_{2,m}b^m \right)e_m+\tilde{\gamma_1} e_{3m}\\
&=&\Big((m\lambda_m^{\pm}-1)(m\lambda_m^{\pm}-m+1)(2m\lambda_m^{\pm}-2m+1)b^{1-m} +(1-m)b^{3m-1} \Big)e_m\\
&+&\tilde{\gamma_1} e_{3m}
\end{eqnarray*}
and
 \begin{eqnarray*}
 \frac{d^2}{dtds} [G_2(t,s,w)]_{|t=0,s=0} &=& \textnormal{Im} \Big\{  b(\varepsilon_{1,2}-\varepsilon_{2,2}) \overline{w}^{3m}+(1-2m)\beta_1 [ v_{1,m} b^{m}-v_{2,m} b^{2m-1} ]w^{m} \Big \} \\
&+&\textnormal{Im} \Big\{(m-1)\beta_2[v_{1,m}b^{m-1}-v_{2m}]\overline{w}^{m} +\beta_2(1-2m) (\theta_{1,2}-\theta_{2,2})\overline{w}^{3m}\\
& +&  (1-\lambda_m^{\pm})\beta_2 v_{2,m}[(1-m)w^m+ (1-2m)\overline{w}^m ]+(1-m)v_{2,m} (\beta_1 b^{2m-1}-\beta_2)w^{m} \Big \} \\
& +&\textnormal{Im} \Big\{ \beta_2(1-2m)(v_{1,m} b^{m-1}-v_{2,m})\overline{w}^{m}+(1-m)v_{2,m}(\widehat{\theta}_{1,2}-\widehat{\theta}_{2,2})\overline{w}^{3m} \Big \} \\
&=&\Big([(m\lambda_m^{\pm}-m+1)\beta_2-m\beta_1b^{2m-1}]v_{2,m}+v_{1,m}b^{m-1}[b \beta_1(2m-1)-m\beta_2] \Big) e_m\\
&+&\tilde{\gamma}_2 e_{3m}\\
&=&\Big( \big((1-m)(m\lambda_m^{\pm}-m+1)-m(2\lambda_m^{\pm}m-2m+1)\big)b^{2m-1}  \\
&+& (m\lambda_m^{\pm}-m+1)b (2m\lambda_m^{\pm}-2m+1)(2m-1) \Big) e_m+\tilde{\gamma}_2 e_{3m}.
\end{eqnarray*}

Using the definition of the projector $Q$, we deduce after some computations
\begin{equation*}
Q \frac{d^2}{dtds}G(t,s,w)_{|t=0,s=0}=\mathcal{K}_m\, \mathbb{W}_m
\end{equation*}
with
\[\mathcal{K}_m \triangleq (2\lambda_m^{\pm}m-2m+1)\frac{b^{1-m} (m\lambda_m^{\pm}-1)(m\lambda_m^{\pm}-m+1)^2+(1-2m)(m\lambda_m^{\pm}-m+1)b^{m+1}+mb^{3m-1}}{[(m\lambda_m^{\pm}-m+1)^2+b^{2m}]^{\frac{1}{2}}}. \]


Eventually, we find 

\[Q \partial_{ff}G(\lambda_m^{\pm},0)[v_m,\partial_{\lambda}\partial_g \varphi(\lambda_m^{\pm},0)v_m]= \widetilde{\beta}_m \mathcal{K}_m\mathbb{W}_m\]

where $\widetilde{\beta}_m$ was defined in \eqref{beta}. This achieves the proof of Proposition \ref{explicit}-$(4)$.
\subsubsection{Computation of $\partial_{ \lambda}\partial_t F_2(\lambda_m^{\pm},0)$}
Now we shall prove the last point of Proposition \ref{explicit}. Recall from Proposition \ref{ImplicitSS} that
\begin{align*}
\partial_{\lambda} \partial_t F_2(\lambda_m^{\pm},0)&=\frac{1}{2}Q\partial_{\lambda} \partial_{ff} G(\lambda_m^{\pm},0)[v_m,v_m]+ \frac{1}{2}Q\partial_{\lambda}\partial_f G(\lambda_m^{\pm},0)(\widehat{v}_m)\\
&+ Q\partial_{ff}G(\lambda_m^{\pm},0)[v_m,\partial_{\lambda}\partial_g \varphi( \lambda_m^{\pm},0)v_m].
\end{align*}

The first term vanishes since
\begin{align*}
\partial_{\lambda} \partial_{ff} G(\lambda_m^{\pm},0)[v_m,v_m]&=\frac{d^2}{dt^2}[ \partial_{\lambda}  G_j(\lambda_m^{\pm},tv_m)]_{| \lambda=\lambda_m^{\pm},t=0}\\
&= -\frac{d^2}{dt^2}_{|t=0} \textnormal{Im} \left\lbrace w \overline{\Phi_j(t,w)}\Phi_j'(t,w) \right\rbrace\\
&=0.
\end{align*}

For the second term we combine \eqref{lin11} with \eqref{vm}
  \[ \partial_{\lambda}\partial_{f}G(\lambda_m^{\pm},0)(\widehat{v}_m)=2m\left( \begin{array}{c}
  \widehat{v}_{1,m} \\ 
  b \widehat{v}_{2,m}
  \end{array} \right) e_{2m}.\]
  Consequently we deduce that 
  \[Q \partial_{\lambda}\partial_{f}G(\lambda_m^{\pm},0)(\widehat{v}_m)=0.\]
Hence we find 
  \[ \partial_{\lambda} \partial_t F_2(\lambda_m^{\pm},0)=Q\partial_{ff}G(\lambda_m^{\pm},0)[v_m,\partial_{\lambda}\partial_g \varphi( \lambda_m^{\pm},0)v_m]. \]
  Now we want to compute 
\[ Q \partial_{ff}G(\lambda_m^{\pm},0)[v_m,\partial_{\lambda}\partial_g \varphi( \lambda_m^{\pm},0)v_m]=Q\partial_t\partial_s[ G(\lambda_m^{\pm},tv_m+s\partial_{\lambda}\partial_g \varphi( \lambda_m^{\pm},0)v_m)]_{|t=0,s=0} .\]
This was done in \cite[Lemma 2-(ii)]{5} combined with \eqref{phiZZ} and \eqref{vm}. We obtain
\[ Q \partial_{ff}G(\lambda_m^{\pm},0)[v_m,\partial_{\lambda}\partial_g \varphi( \lambda_m^{\pm},0)v_m]=0\]
and this completes the proof of the desired result.

\section{Proof of the main theorem}
This section is devoted to the proof of  Theorem $\ref{main}$. To begin, we choose a small neighborhood of $(\lambda_m^+,0)$ in the strong topology of $\R\times X_m$ such that the equation 
$$
F_1(\lambda,  t v_m,k)=0, \quad  k\in   \mathcal{X}_m
$$
admits locally  a unique surface  of solutions parametrized  by 
\[(\lambda,t)\in(\lambda_m^+-\epsilon_0,\lambda_m^+ +\epsilon_0)\times (-\epsilon_0,\epsilon_0)\mapsto k=\varphi(\lambda,t v_m)\in \mathcal{X}_m,
\] with $\epsilon_0>0$ and $\varphi$ being a $C^1$ function and  actually it is of class   $C^k$ for any $k\in \N$. This follows from the fact that the functionals defining the V-states are better than $C^1$ and it could be proved that they are in fact of class $C^k$. Note also  that 
$$
\varphi(\lambda_m^+,0)=0.
$$For more details we refer to the subsection \ref{Linearized operator and Lyapunov-Schmidt reduction}. We recall from that subsection that  the V-states equation is equivalent to,
\begin{equation}\label{F_2}
F_2(\lambda,t)=0 
\end{equation}
with  $(\lambda,t)$ being  in the neighborhood of $(\lambda_m^+,0)$ in $\R^2$ and $F_2:\R^2\to \langle\mathbb{W}_m\rangle$. We intend to prove the following assertion: there exists $b_m\in(0,b_m^*)$ such that for any $b\in (b_m, b_m^*)$ there exists $\varepsilon>0$ such  that the set 
\begin{equation}\label{param23}
\mathcal{E}_b\triangleq\big\{(\lambda,t)\in (\lambda_m^+-\varepsilon, \lambda_m^+ +\varepsilon)\times(-\varepsilon,\varepsilon),\, F_2(\lambda,t)=0\big\}
\end{equation}
 is a $C^1$-Jordan  curve in the complex plane.
 Taylor expansion of $F_2$ around the point $(\lambda_m^+,0)$ at the order two is given by,
\begin{eqnarray*}F_2(\lambda,t)&=&\partial_{\lambda}F_2(\lambda_m^+,0)(\lambda-\lambda_m^+)+\partial_{t}F_2(\lambda_m^+,0) t\\
&+&\frac{1}{2}\partial_{\lambda \lambda} F_2(\lambda_m^+,0)(\lambda-\lambda_m^+)^2+\frac{1}{2}\partial_{tt}F_2(\lambda_m^+,0)t^2\\
&+&((\lambda-\lambda_m^+)^2+t^2)\epsilon(\lambda,t)
\end{eqnarray*}
where 
\[\underset{(\lambda,t) \rightarrow (\lambda_m^+,0)}{\lim} \epsilon(\lambda,t)=0.\]
  Then using Proposition $\ref{explicit}$ we get for any $( \lambda,t)\in (\lambda_m^+-\eta,\lambda_m^++\eta) \times(-\eta,\eta)$, with $\eta>0$,
\[ F_2(\lambda,t)=-\Big(a_m(b)(\lambda-\lambda_m^+)+c_m(b)(\lambda-\lambda_m^+)^2+d_m(b) t^2+\big((\lambda-\lambda_m)^2+t^2\big)\epsilon(\lambda,t)\Big) \mathbb{W}_m\]
with 
\begin{eqnarray*}a_m(b)&=&-\frac{m[(m\lambda_m^+-m+1)^2-b^{2m}]}{b^{m-1}[(m\lambda_m^+ -m+1)^2+b^{2m}]^{\frac{1}{2}}},\\
c_m(b)&=&-\frac{2m^2b^{1-m} (m\lambda_m^+-m+1)^3}{[(m\lambda_m+-m+1)^2+b^{2m}]^{\frac{3}{2}}} 
\end{eqnarray*}
and
\[d_m(b)=\frac{m}{2}(m-1)b^{3-3m} \frac{(b^{2m-2}-(m\lambda_m^+-m+1)^2)^2}{([m\lambda_m^+-m+1]^2+b^{2m})^{\frac{1}{2}}} -\frac{\widetilde{\beta}_m}{2} \mathcal{K}_m.\]
Note that for $b=b_m^*$ we have $\Delta_m=0$ which implies that
$$\lambda_m^+=\frac{1+b^2}{2},\quad m\lambda_m^+-m+1=-b^m.$$
Thus we get
$$a_m(b_m^*)=0, c_m(b_m^*)>0.
$$
Moreover we can check that for $b=b_m^*$, 
$$
\widetilde{\beta}_m>0 \quad \textnormal{and}\quad  \mathcal{K}_m<0
$$
which implies in turn that $d_m(b_m^*)>0.$ Those properties on the signs  remain true for $b$ close to $b_m^*$, that is, $b$ belongs to some interval  $ [b_m, b_m^*]$. In addition we deduce  from the identity given in Remark \ref{rmq7} that for any $b\in (0,b_m^*)$ we have $a_m(b)>0.$ Indeed,
\begin{eqnarray*}
(m\lambda_m^+-m+1)^2-b^{2m}&=&\big(m\lambda_m^+-m+1-b^{m}\big)\big(m\lambda_m^+-m+1+b^{m}\big)\\
&=&\big(\sqrt{\Delta_m}-\sqrt{\Delta_m+ b^{2m}}-b^m\big)\big(\sqrt{\Delta_m}+b^m-\sqrt{\Delta_m+ b^{2m}}\big)\\
&<&0.
\end{eqnarray*}
Set $x_0(b)=\frac{a_m(b)}{2c_m(b)}$ and using the change of variables
\[s=\lambda-\lambda_m^++x_0(b),\quad \psi(s,t)\triangleq\epsilon\big(s+\lambda_m^+-x_0(b),t\big)\]
then the equation of $F_2$ becomes
\begin{equation}\label{ellipse1}c_m(b) s^2+d_m(b) t^2-\frac{a_m^2}{4 c_m(b)}+\big((s-x_0(b))^2+t^2\big)\psi(s,t)=0
\end{equation}
and 
\[
\underset{(s,t) \rightarrow (x_0(b),0)}{\lim} \psi(s,t)=0.\]
Note that if we remove $\psi$ from the second term of this equation we get the equation of a small ellipse centered at $(0,0)$ and of semi-axes $\frac{a_m(b)}{2c_m(b)}$ and $\frac{a_m(b)}{2\sqrt{d_m(b) c_m(b)}}$.  Thus taking $b$ close enough to $b_m^*$ one can guarantee that this ellipse is contained in the box $(-\epsilon_0,\epsilon_0)^2$ for  which the solutions of the equation $F_1$ still parametrized by $\varphi$. By small perturbation we expect to get a curve of solutions to $F_2$ which is a small perturbation of the ellipse. To prove rigorously  this expectation  we start with the  change of variable,
\[ t=x_0(b)\sqrt{\frac{c_m(b)}{d_m(b)}}\,x \quad \text{ and } \quad s=x_0(b)y.\]
Consequently,  the equation \eqref{ellipse1} becomes
\begin{equation}\label{Eqwa1}
\mathcal{G}(b,x,y)\triangleq x^2+y^2-1+\frac{1}{c_m(b)}\Big((y-1)^2+\frac{c_m(b)^2}{d_m(b)}x^2\Big)\widehat{\psi}\big(x_0(b),x,y\big)=0
\end{equation}
with
\[\widehat{\psi}(\mu,x,y)\triangleq \epsilon\Bigg( \mu\, y+\lambda_m^+-\mu, \mu\, \sqrt{\frac{c_m(b)}{d_m(b)}}x\Bigg).\]
 Now we shall characterize the geometric structure of  the planar set 
 \[ \widehat{\mathcal{E}}_b\triangleq  \big\{ (x,y)\in \R^2;\,\, \mathcal{G}(b,x,y)=0 \big\}.\]
 \begin{lemma}\label{Teclem}
 There exists $b_m\in(0,b_m^*)$ such that for any $b\in(b_m, b_m^*)$ the set  $\widehat{\mathcal{E}}_b$  contains  a ${C}^1$-Jordan curve and the point $(0,0)$ is located inside. In addition, this curve is a smooth  perturbation of the unit circle $\mathbb{T}$.
 
 \end{lemma}
 \begin{proof}
 We shall look for  a curve of solutions lying in  the set $\widehat{\mathcal{E}}_b$ and  that can be parametrized through   polar coordinates  as follows $$
\theta\in [0,2\pi]\mapsto (x,y)=R(\theta) e^{\num{i}\theta}.
 $$
Now we fix $b$ and  introduce the function
\begin{eqnarray*}
\mathcal{F}(\mu, R(\theta))&\triangleq&R^2(\theta)-1\\
&-&\frac{1}{c_m(b)}\left([R(\theta)\sin \theta-1]^2+\frac{c_m(b)^2}{d_m(b)}R^2(\theta) \cos^2\theta\right)\widehat{\psi}\Big(\mu, R(\theta)\cos \theta, R(\theta)\sin\theta\Big).
\end{eqnarray*}
Then according to   \eqref{Eqwa1} it is enough to solve 
\begin{equation}\label{Eqwa2}
\mathcal{F}(\mu, R(\theta))=0\quad \hbox{and}\quad \mu=x_0(b).
\end{equation}
Recall that the function  $\epsilon$ is defined in the box $( \lambda,t)\in (\lambda_m^+-\eta,\lambda_m^++\eta) \times(-\eta,\eta)$ for some given real number  $\eta>0.$ Thus it is not hard to find an implicit value   $\mu_0>$ such that 
\[\mathcal{F}:  (-\mu_0, \mu_0) \times  \mathcal{B} \to  {C}^1(\mathbb{T})\]
is well-defined and is of  class $C^1$,  with $\mathcal{B}$ being  the open set of $C^1(\T)$ defined by
\[
 \mathcal{B}=\Big\{ R \in C^1(\T);\,\, \Vert 1-R\Vert_{\infty}+\|R^\prime\|_{\infty} < \frac{1}{2} \Big\}.
\]
In addition 
 \[ \mathcal{F}(0,1)=0\quad\textnormal{and}\quad \partial_R \mathcal{F}(0,1)h=2h,\,\forall h\in C^1(\T).\]
Thus $\partial_R \mathcal{F}(0,1)$ is an isomorphism and by  the implicit function theorem we deduce the existence of $\mu_1>0$ such that for any $\mu\in(-\mu_1,\mu_1)$ there exists a unique $R_\mu\in \mathcal{B}$ such that $(\mu,R_\mu)$ is a solution for $\mathcal{F}(\mu, R_\mu)=0$. It is worthy to point out that $\mu_1$ can be chosen independent of $b\in[b_m,b_m^*]$ because the coefficients $\frac{1}{c_m(b)}$ and $\frac{d_m(b)}{c_m(b)}$ that appear in the nonlinear contribution of $\mathcal{F}$ are  bounded.   Now since $\lim_{b\to b_m^*}x_0(b)=0$ then we deduce the existence of $b_m\in(0,b_m^*)$ such that for any $b\in(b_m, b_m^*)$ the set of solutions for \eqref{Eqwa2} is described around the trivial solution $R=1$ (the circle) by the curve $\theta\in [0,2\pi]\mapsto R_{x_0(b)}(\theta)e^{\num i\theta}$. 
 On the other hand from the definition of $\mathcal{B}$ we deduce that this  curve of solutions is contained in the annulus centered at the origin and of radii $\frac12$ and $\frac32$. It should be also non self-intersecting $C^1$ loop according to the regularity  of the polar parametrization. This achieves the proof of the lemma.

\end{proof}
Now let us see how to use Lemma \ref{Teclem}  to end  the proof of Theorem \ref{main}. We have already  seen in the subsection \ref{Linearized operator and Lyapunov-Schmidt reduction} that the vectorial conformal mapping $\Phi\triangleq\left( \begin{array}{c}
\Phi_1 \\ 
\Phi_2
\end{array}  \right)$ that describes the bifurcation curve from $(\lambda_m^+,0)$ is decomposed as follows
$$
\Phi(w)=\left( \begin{array}{c}
1 \\ 
b
\end{array}  \right)w+tv_m+\varphi(\lambda,tv_m)\quad\textnormal{with}\quad (t,\lambda)\in \mathcal{E}_b,
$$
where $ \mathcal{E}_b,$ has been defined in \eqref{param23}. As the relationship between the sets $\mathcal{E}_b$ and $\widehat{\mathcal{E}}_b$ is given through a non degenerate affine  transformations then the set $\mathcal{E}_b$ is also a $C^1$-Jordan curve. In addition the point $(\lambda_m^+ -\frac{a_m(b)}{2 c_m(b)},0) $ is located inside the curve and recall also that the point $(\lambda_m^+,0)$ belongs to $\mathcal{E}_b.$ This fact implies that the curve intersects necessary the real axis on  another point $(\lambda_m,0)\neq(\lambda_m^+,0)$. By virtue of  \eqref{phitriv}  the conformal mappings at those two points coincide  which implies in turn  that the bifurcation curve bifurcates also from the trivial solution at  the point $(\lambda_m,0)$. Note that looking from the side $(\lambda_m,0)$ this curve represents V-states with exactly $m$-fold symmetry and not with more symmetry; to be convinced see the structure of $\Phi$.  However from the local bifurcation diagram close to the trivial solution which was   studied in  \cite{3} we know that the only V-states  bifurcating from the trivial solutions with $m$-fold symmetry bifurcate from the points $(\lambda_m^+,0)$ and $(\lambda_m^-,0)$. Consequently we get $\lambda_m=\lambda_m^-$ and this shows that the bifurcation curves of the $m-$fold V-states merge and form a small loop. 

The last point to prove concerns the symmetry of the curve $\mathcal{E}_b$ with respect to the $\lambda$ axis. This reduces  to  check that for $(\lambda,t)\in \mathcal{E}_b$ then 
$(\lambda,-t)\in \mathcal{E}_b$.
Indeed, if $D=D_1\setminus D_2$ is an $m$-fold doubly-connected V-state, its boundaries are parametrized by the conformal mappings
\[ \Phi_j(w)= w \left(b_j+ a_{j,1} \,\overline{w}^m+ \sum_{n \geq 2} {a_{j,n}}{\overline{w}^{mn}} \right), \, a_{j,n} \in \R \]
Hence we find that the vectorial conform mapping $\Phi$ admits the decomposition 
$$\Phi(w)= \left( \begin{array}{c}
1 \\ 
b
\end{array} \right) w+ \left( \begin{array}{c}
a_{1,1} \\ 
a_{2,1}
\end{array} \right)\overline{w}^{m-1}+\Psi(w),\quad \Psi \in \mathcal{X}_m
$$
Observe that we have a unique $t$ such that
\begin{equation}\label{confor45}
 \left( \begin{array}{c}
a_{1,1} \\ 
a_{2,1}
\end{array} \right)\overline{w}^{m-1}= t v_m(w)+ \Psi_1(w),\quad  \Psi_1\in \mathcal{X}_m
\end{equation}
Now we consider $\widehat{D}=e^{\frac{i\pi}{m}} D$, the rotation of D with the angle $\frac{\pi}{m}$. Then the new domain is also a V-state with a $m$-fold symmetry rotating with the same angular velocity. Thus it should be associated to a point $(\lambda, \tilde{t})\in \mathcal{E}_b$. Now the  conformal parametrization of $\widehat{D}$ is given by 
\[ \widehat{\Phi_j}(w)= e^{-\frac{i\pi}{m}} \Phi_j(e^{\frac{i\pi}{m}} w)=w \left( b_j+ \sum_{n \geq 1}(-1)^n \frac{a_{j,n}}{w^{mn}} \right), \, a_{j,n} \in \R. \]
Thus the vectorial conformal parametrization $\widehat{\Phi}\triangleq\left( \begin{array}{c}
\widehat{\Phi}_1 \\ 
\widehat{\Phi}_2
\end{array}  \right)$ admits the decomposition 

$$\widehat{\Phi}(w)= \left( \begin{array}{c}
1 \\ 
b
\end{array} \right) w- \left( \begin{array}{c}
a_{1,1} \\ 
a_{2,1}
\end{array} \right)\overline{w}^{m-1}+\widehat{\Psi}(w),\quad \widehat\Psi \in \mathcal{X}_m
$$
In view of \eqref{confor45} we deduce that
\begin{equation}\label{confor45}
\widehat{\Phi}(w)= \left( \begin{array}{c}
1 \\ 
b
\end{array} \right) w-t v_m(w)+\widehat{\Psi_1}(w),\quad \widehat{\Psi_1}(w)\triangleq - \Psi_1+\widehat\Psi \in \mathcal{X}_m
\end{equation}
This shows that $(\lambda,-t)$ belongs to the curve $\mathcal{E}_b$ and this concludes the desired result.

\begin{ackname}
 The authors are partially supported by the ANR project Dyficolti ANR-13-BS01-0003- 01.
  \end{ackname}
%

  \end{document}